\pdfoutput=1
\RequirePackage{ifpdf}
\ifpdf 
\documentclass[pdftex]{sigma}
\else
\documentclass{sigma}
\fi

\usepackage{amsbsy,amscd,stmaryrd,enumerate}
\usepackage[bbgreekl]{mathbbol}
\usepackage[mathscr]{euscript}
\usepackage{esint}
\usepackage{upgreek}
\usepackage{mathcomp}
\usepackage{titlesec}
\titleformat{\subsection}[runin]{\normalfont\bfseries}{\thesubsection.}{3pt}{}

\newcommand{\jl}{(\!(}
\newcommand{\jr}{)\!)}
\newcommand{\spl}{\{\!\!\{}
\newcommand{\rpl}{\}\!\!\}}
\newcommand{\shl}{\llbracket}
\newcommand{\shr}{\rrbracket}
\newcommand{\laa}{\langle\!\langle}
\newcommand{\raa}{\rangle\!\rangle}
\newcommand{\Gf}{\mathcal{G}}
\newcommand{\hamsec}{\mathsf{H}}
\newcommand{\hhamsec}{\widehat{\hamsec}}
\newcommand{\hinf}{\mathsf{h}}
\newcommand{\ssec}{\mathsf{S}}
\newcommand{\hssec}{\widehat{\ssec}}
\newcommand{\hzsec}{\widehat{\mathsf{Z}}}
\newcommand{\hbsec}{\widehat{\mathsf{B}}}
\newcommand{\hS}{\widehat{\mathcal{S}}}
\newcommand{\hge}{\mathfrak{h}}
\newcommand{\sag}{\widehat{\hge}}
\newcommand{\twprod}{\star}
\newcommand{\imt}{\iota}
\newcommand{\mcoc}{\Sigma}
\newcommand{\mom}{\mathbb{M}}
\newcommand{\fc}{\Gamma}
\newcommand{\X}{\mathscr{X}}
\newcommand{\lge}{\mathscr{L}}
\newcommand{\gge}{\mathscr{G}}
\newcommand{\sfr}{\mathscr{F}}
\newcommand{\trun}{\operatorname{\top}}
\newcommand{\lop}{\mathscr{L}}
\newcommand{\dn}{d_{\nabla}}
\newcommand{\sd}{\delta}
\newcommand{\rice}{\mathsf{Rc}}
\newcommand{\dop}{\mathscr{P}}
\newcommand{\Wl}{\operatorname{W}}
\newcommand{\pon}{\mathsf{p}}
\newcommand{\curv}{\mathcal{C}}
\newcommand{\ex}{\mathscr{E}}
\newcommand{\sRo}{\mathscr{S}}
\newcommand{\bric}{\overline{\ric}}
\newcommand{\scal}{\mathscr{R}}
\newcommand{\sOmega}{\mathbb{\Omega}}
\renewcommand{\H}{\mathscr{H}}
\newcommand{\tcon}{\mathbb{T}}
\newcommand{\shm}{\mathbb{H}}
\newcommand{\hm}{\mathsf{H}}
\newcommand{\sOm}{\mathbb{\Omega}}
\newcommand{\symcon}{\mathbb{S}}
\newcommand{\Ham}{\operatorname{Ham}}
\newcommand{\Symplecto}{\operatorname{Symp}}
\newcommand{\ham}{\operatorname{ham}}
\newcommand{\symplecto}{\operatorname{symp}}
\newcommand{\Om}{\Omega}
\newcommand{\con}{\mathsf{c}}
\newcommand{\sflat}{\curlyvee}
\newcommand{\ssharp}{\curlywedge}
\newcommand{\sprod}{\odot}
\newcommand{\pol}{\mathsf{Pol}}
\newcommand{\vr}{\updelta}
\newcommand{\sR}{\mathscr{R}}
\newcommand{\vol}{\mathsf{vol}}
\newcommand{\ka}{\kappa}
\newcommand{\T}{\mathcal{T}}
\newcommand{\Id}{\operatorname{Id}}
\newcommand{\dum}{\,\cdot\,}
\newcommand{\Ga}{\Gamma}
\newcommand{\cycle}{\operatorname{Cycle}}
\newcommand{\cyclearg}[1]{\underset{#1}{\cycle}}
\newcommand{\ric}{\mathsf{Ric}}
\newcommand{\Q}{\mathcal{Q}}
\renewcommand{\j}{\mathfrak{j}}
\newcommand{\hj}{\hat{\j}}
\newcommand{\hop}{\mathscr{H}}
\newcommand{\reat}{\mathbb{R}^{\times}}
\newcommand{\ext}{\Lambda}
\newcommand{\rc}{\mathscr{R}}
\newcommand{\cinf}{C^{\infty}}
\newcommand{\eno}{\operatorname{End}}
\newcommand{\ctm}{T^{\ast}M}
\newcommand{\si}{\sigma}
\newcommand{\lin}{\mathsf{L}}
\newcommand{\bnabla}{\bar{\nabla}}
\newcommand{\integer}{\mathbb{Z}}
\newcommand{\lie}{\mathfrak{L}}
\newcommand{\Aff}{\mathbb{Aff}}
\newcommand{\aff}{\mathbb{aff}}
\newcommand{\ste}{\mathbb{V}}
\newcommand{\F}{\mathcal{F}}
\newcommand{\tf}{\mathbb{\Theta}}
\newcommand{\A}{\mathcal{A}}
\newcommand{\std}{\mathbb{V}^{\ast}}
\newcommand{\stw}{\mathbb{W}}
\newcommand{\lb}{\langle}
\newcommand{\ra}{\rangle}
\newcommand{\al}{\alpha}
\newcommand{\be}{\beta}
\newcommand{\ga}{\gamma}
\newcommand{\G}{\widehat{G}}
\newcommand{\emf}{\mathcal{E}}
\newcommand{\g}{\mathfrak{g}}
\newcommand{\tensor}{\otimes}
\newcommand{\Ad}{\operatorname{Ad}}
\newcommand{\K}{\mathscr{K}}
\newcommand{\rea}{\mathbb R}
\newcommand{\hemom}{\widehat{\mom}}
\newcommand{\henon}{\mathbb{N}}
\newcommand{\tr}{\operatorname{\mathsf{tr}}}
\newcommand{\emom}{\mathscr{M}}
\renewcommand{\P}{\mathcal{P}}

\numberwithin{equation}{section}

\newtheorem{Theorem}{Theorem}[section]
\newtheorem{Corollary}[Theorem]{Corollary}
\newtheorem{Lemma}[Theorem]{Lemma}
 { \theoremstyle{definition}
\newtheorem{Remark}[Theorem]{Remark} }

\begin{document}


\newcommand{\arXivNumber}{1606.07739}

\renewcommand{\PaperNumber}{002}

\FirstPageHeading

\ShortArticleName{Symmetries of the Space of Linear Symplectic Connections}

\ArticleName{Symmetries of the Space\\ of Linear Symplectic Connections}

\Author{Daniel J.F.~FOX}

\AuthorNameForHeading{D.J.F.~Fox}

\Address{Departamento de Matem\'aticas del \'Area Industrial, Escuela T\'ecnica Superior de Ingenier\'{\i}a\\
y Dise\~no Industrial, Universidad Polit\'ecnica de Madrid, Ronda de Valencia 3,\\ 28012 Madrid, Spain}
\Email{\href{mailto:daniel.fox@upm.es}{daniel.fox@upm.es}}

\ArticleDates{Received June 30, 2016, in f\/inal form January 07, 2017; Published online January 10, 2017}

\Abstract{There is constructed a family of Lie algebras that act in a Hamiltonian way on the symplectic af\/f\/ine space of linear symplectic connections on a symplectic manifold. The associated equivariant moment map is a formal sum of the Cahen--Gutt moment map, the Ricci tensor, and a translational term. The critical points of a functional constructed from it interpolate between the equations for preferred symplectic connections and the equations for critical symplectic connections. The commutative algebra of formal sums of symmetric tensors on a symplectic manifold carries a pair of compatible Poisson structures, one induced from the canonical Poisson bracket on the space of functions on the cotangent bundle polynomial in the f\/ibers, and the other induced from the algebraic f\/iberwise Schouten bracket on the symmetric algebra of each f\/iber of the cotangent bundle. These structures are shown to be compatible, and the required Lie algebras are constructed as central extensions of their linear combinations restricted to formal sums of symmetric tensors whose f\/irst order term is a multiple of the dif\/ferential of its zeroth order term.}


\Classification{53D20; 53D05; 53C05; 17B99}

\setcounter{tocdepth}{1}

\section{Introduction}
A (linear) \textit{symplectic connection} on a symplectic manifold $(M, \Om)$ is a torsion-free af\/f\/ine connection preserving the symplectic form~$\Om$.
The space $\symcon(M, \Om)$ of symplectic connections is a~symplectic af\/f\/ine space; the dif\/ference of two symplectic connections is identif\/ied with a section of the bundle $S^{3}(\ctm)$ of completely symmetric covariant three tensors, and the symplectic form $\sOm$ on $\symcon(M, \Om)$ is given by integrating over $M$ the pairing induced on such sections by $\Om$. The geometry of the symplectic af\/f\/ine space $(\symcon(M, \Om), \sOm)$ is the focus of this note.

Ef\/forts have been made to identify geometrically interesting classes of symplectic connections analogous to classes of connections studied in the Riemannian setting, such as the Levi-Civita connections of Einstein or constant scalar curvature metrics. Such classes are def\/ined in terms of the zeros and critical points of functionals on $(\symcon(M, \Om), \sOm)$ equivariant or invariant with respect to some action of some group of symplectomorphisms or Hamiltonian dif\/feomorphisms.

The geometry of any symplectic space is ref\/lected in its Poisson algebra of functions. Interesting functionals on $(\symcon(M, \Om), \sOm)$ are constructed from the curvature of $\nabla \in \symcon(M, \Om)$. The simplest interesting classes of symplectic connections that have been studied are those of Ricci type, the preferred symplectic connections (these include the Ricci f\/lat symplectic connections), and the critical symplectic connections. The focus here is on preferred and critical symplectic connections.

The critical points of the integral of the trace of the square of the Ricci endomorphism are the \textit{preferred} symplectic connections f\/irst studied by F.~Bourgeois and M.~Cahen \cite{Bourgeois-Cahen-preferred, Bourgeois-Cahen}. A~symplectic connection is preferred if and only if there vanishes the complete symmetrization of the covariant derivative of its Ricci tensor $\ric_{ij} = R_{ij}$, that is, it solves $(\sd^{\ast}\ric)_{ijk} = -\nabla_{(i}R_{jk)} = 0$. The functionals given by integrating the trace of higher even powers of the Ricci endomorphism play a role in the averaging procedure used by Fedosov in~\cite{Fedosov-atiyahbottpatodi}, and, as is explained brief\/ly in Section~\ref{preferredsection}, their critical points generalize preferred symplectic connections. However, because the f\/irst Pontryagin class of a symplectic manifold equals the integral of an expression quadratic in the curvature of the connection, the class of preferred symplectic connections is the only interesting class of symplectic connections obtained as the critical points of a functional given by integrating expressions quadratic in the curvature tensor (this observation is due to Proposition~2.4 of~\cite{Bourgeois-Cahen}; see also the introduction of~\cite{Fox-critical}).

\looseness=-1 In \cite{Cahen-Gutt}, M.~Cahen and S.~Gutt showed that the action of the group of Hamiltonian dif\/feo\-mor\-phisms on $\symcon(M, \Om)$ is Hamiltonian, with a moment map denoted here by $\K(\nabla)$ (see~\eqref{kdefined} for its def\/inition). In~\cite{Fox-critical}, the author began the study of the \textit{critical} symplectic connections, that are def\/ined to be the critical points of the integral $\emf(\nabla) = \int_{M}\K(\nabla)^{2}\Om_{n}$ (where $\Om_{n} = \tfrac{1}{n!}\Om^{n}$) with respect to arbitrary variations of~$\nabla$. A symplectic connection $\nabla$ is critical if and only if the Hamiltonian vector f\/ield $\hm_{\K(\nabla)}$ generated by~$\K(\nabla)$ is an inf\/initesimal automorphism of $\nabla$. Explicitly this means that $\hop(\K(\nabla)) = 0$ where $\hop(f)_{ijk} = (\lie_{\hm_{f}}\nabla)_{ij}\,^{p}\Om_{pk}$, $\hm_{f}$ is the Hamiltonian vector f\/ield generated by $f \in \cinf(M)$, and $\lie_{\hm_{f}}$ is the Lie derivative along $\hm_{f}$. The moment map~$\K(\nabla)$ is the sum of a~multiple of the twofold divergence of the Ricci tensor of $\nabla$ and a multiple of the complete contraction of the f\/irst Pontryagin form of $\nabla$ with $\Om \wedge \Om$. By the main theorem of D.~Tamarkin~\cite{Tamarkin-symplecticconnections}, a functional on $\symcon(M, \Om)$ given by integrating against $\Om_{n}$ polynomials in the curvature tensor and its covariant derivatives contracted with the symplectic form is independent of the choice of connection if and only if its integrand is, modulo divergences, a polynomial in Pontryagin forms contracted with the symplectic form. It follows that~$\K$ is in some sense, not made precise here, the simplest symplectomorphism equivariant map from $\symcon(M, \Om)$ to~$\cinf(M)$.

Suppose $(M, \Om)$ is compact, and f\/ix a reference symplectic connection $\nabla_{0}$. Here there are constructed, for each $(s, t) \in \rea \times \reat$, a Lie bracket $\jl \dum, \dum \jr_{s, t}$ on
\begin{gather}\label{sagdef}
\sag = \cinf(M) \oplus \Ga\big(S^{2}(\ctm)\big) \oplus \Ga\big(S^{3}(\ctm)\big)
\end{gather}
and \looseness=-1 a Hamiltonian action of the Lie algebra $(\sag, \jl \dum, \dum \jr_{s, t})$ on $(\symcon(M, \Om), \sOm)$. The associated moment map is essentially a linear combination of $\K$ and the Ricci tensor (see~\eqref{hemomfull} for the precise expression when $(s, t) = (1, 1)$; it requires some preliminary discussion to make sense of it). The choice of $\nabla_{0}$ is unimportant in the sense that the Lie algebras associated with dif\/ferent choices of $\nabla_{0}$ are isomorphic. In fact, the brackets $\jl\dum,\dum\jr_{s, t}$ (and corresponding Hamiltonian actions) for which neither $s$ nor $t$ is $0$ are all pairwise isomorphic, so isomorphic to $\jl\dum,\dum\jr = \jl\dum,\dum\jr_{1, 1}$; the reason for considering the full family is that it interpolates between the degenerate cases $(s, t) =(1, 0)$ and $(s, t) = (0, 1)$ (this is discussed further below). In~\eqref{sagdef}, the factor $\cinf(M)$ should be seen as the central extension of the Lie algebra $\ham(M, \Om)$ of Hamiltonian vector f\/ields, while $\Ga(S^{2}(\ctm))$ should be seen as the Lie algebra of inf\/initesimal gauge transformations of the symplectic frame bundle. The Lie bracket $\jl\dum,\dum\jr$ combines and extends these actions.

The somewhat complicated def\/inition of the bracket $\jl \dum, \dum \jr$ is summarized now. The precise def\/inition is~\eqref{stbracketdefined} of Section~\ref{maintheoremsection}; see also Theorem~\ref{emomtheorem}. The algebra $\hssec(M)$ whose elements are formal linear sums of completely symmetric tensors on $(M, \Om)$ (see Section~\ref{symmetrictensorsection} for the precise meaning) carries two Lie brackets, a dif\/ferential Lie bracket $\shl\dum, \dum\shr$ induced from the Schouten bracket on symmetric contravariant tensors on $M$ and an algebraic Lie bracket $(\dum, \dum)$ induced f\/iberwise by regarding the symmetric tensor algebra as the associated graded algebra of the Weyl algebra. The def\/inition of $\jl \dum, \dum \jr$ requires Lemma~\ref{compatibilitylemma}, showing that these Lie brackets are compatible in the sense of bi-Hamiltonian systems; precisely, the dif\/ferential, in the Lie algebra cohomology of $(\dum, \dum)$, of the symmetrized covariant derivative along $\nabla$ equals the Schouten bracket $\shl\dum,\dum\shr$ (see~\eqref{twobrackets} of Lemma~\ref{compatibilitylemma}). (Although it seems that this claim about the Schouten bracket must be known in some form to experts, I do not know a reference.) This means that any linear combination $[\dum, \dum]_{s, t} = s(\dum, \dum) + t\shl\dum,\dum\shr$ is again a Lie bracket.

\looseness=-1 The space $\hssec(M)$ is too big because its $1$-graded part comprises all one-forms on $M$, rather than the closed one-forms, that are dual to symplectic vector f\/ields. Because the $[\dum, \dum] = [\dum, \dum]_{1, 1}$ bracket of closed one forms in $\hssec(M)$ need not be closed, it is inadequate to restrict to the subspace of $\hssec(M)$ the $1$-graded part of which comprises closed one-forms (although something can be done in this directions; see Lemma~\ref{hamseccentrallemma}). However, the subspace $\hhamsec(M)$, comprising formal sums $\al_{0} + \al_{1} + \cdots$ such that $\al_{k} \in \Ga(S^{k}(\ctm))$ and $\al_{1} = -d\al_{0}$ is a subalgebra of $(\hssec(M), [\dum, \dum])$, and this subalgebra plays a prominent role in the discussion. Lemmas~\ref{hamseclemma} and~\ref{hamseccentrallemma} describe its structure.

Although $(\hssec(M), [\dum, \dum])$ cannot be made to act on $(\symcon(M, \Om), \sOm)$ in any obvious way, the subalgebra $\hhamsec(M)$ admits a symplectic af\/f\/ine action on $(\symcon(M, \Om), \sOm)$ such that the ideal $\hhamsec_{4}(M) = \prod_{k \geq 4}\Ga(S^{k}(\ctm))$ acts trivially. This action is weakly Hamiltonian, but the cocycle obstructing equivariance can be computed explicitly. This yields a central extension that acts in a Hamiltonian manner on $(\symcon(M, \Om), \sOm)$, with moment map $\hemom$. Since the Lie ideal $\hhamsec_{4}(M)$ acts trivially, quotienting by it yields an action of the Lie algebra $(\sag, \jl\dum,\dum\jr)$.

The action of the group of gauge transformations of the symplectic frame bundle of $M$ on the space of af\/f\/ine connections, possibly having torsion, preserving $\Om$, is Hamiltonian with a moment map constructed in the manner of Atiyah and Bott in \cite{Atiyah-Bott}. See Theorem~\ref{abtheorem} for the precise statement. Using the symplectic form the Lie algebra of inf\/initesimal gauge transformations can be identif\/ied with the space $\Ga(S^{2}(\ctm))$ equipped with a constant multiple of the algebraic Lie bracket mentioned above. The obstruction to this Lie algebra acting on $\symcon(M, \Om)$ is given by the Schouten bracket. However, an extended bracket on $\Ga(S^{2}(\ctm)) \oplus \Ga(S^{3}(\ctm))$ built from the Poisson bracket and the algebraic bracket does act on $\symcon(M, \Om)$, by symplectic af\/f\/ine transformations. This extended action can be combined in a coherent way with the action of $\ham(M, \Om)$ on $\symcon(M, \Om)$. While the resulting action is only weakly Hamiltonian, the cocycle measuring nonequivariance of the moment map can be computed explicitly and from it there is constructed the Hamiltonian action of $(\sag, \jl \dum, \dum \jr_{s, t})$. This action combines, at the inf\/initesimal level, the actions of $\Ham(M, \Om)$ and the group of gauge transformations to yield a Hamiltonian action with a moment map $\hemom$ that simultaneously extends the moment map $\K$ for the action of the Lie algebra $\ham(M, \Om)$ of Hamiltonian vector f\/ields and the Atiyah--Bott moment map.

When $M$ is compact, there is constructed (see~\eqref{engstdefined} for the precise def\/inition) from $\hemom$ a functional $\henon\colon \symcon(M, \Om) \to \rea$. There is a graded symmetric bilinear pairing on $\hhamsec(M)$ that descends to $\sag$. Pairing the moment map $\hemom$ with itself with respect to this pairing yields $\henon$. The critical points of $\henon$ are the solutions of
\begin{gather}\label{coupled}
t^{2}\hop(\K(\nabla)) + s^{2}\sd^{\ast}\ric = 0.
\end{gather}
The equations~\eqref{coupled} interpolate between the equations for preferred and critical symplectic connections. The case $(s, t) = (1, 0)$ yields the equations for the preferred symplectic connections introduced by Bourgeois and Cahen \cite{Bourgeois-Cahen}, and the case $(s, t) = (0, 1)$ recovers the equations for the critical symplectic connections introduced by the author in \cite{Fox-critical}.
\begin{Remark}
In \cite{Fox-critical} it is shown that on a $2$-manifold a preferred symplectic connection is critical. Consequently, on a $2$-dimensional symplectic manifold a preferred symplectic connection solves~\eqref{coupled} for any $(s, t) \in \rea^{2}$.
\end{Remark}

Section~\ref{symmetrictensorsection} gives the background needed to formulate and prove Lemma~\ref{compatibilitylemma}. Section~\ref{functionalsection} recalls the def\/initions and basic characterizations of critical and preferred symplectic connections. Section~\ref{atiyahbottsection} discusses the failure of the symplectomorphism group to act in a Hamiltonian manner on $\symcon(M, \Om)$. Finally, the construction of $\jl \dum, \dum \jr_{s, t}$ and its action on $\symcon(M, \Om)$ is given in Section~\ref{symplecticaffinesection}.

\section{Preliminaries}\label{backgroundsection}
\subsection{}
Smooth means $C^{\infty}$, and all manifolds, bundles, sections, etc. are smooth, unless otherwise indicated. Throughout, $(M, \Omega)$ is a connected $2n$-dimensional symplectic manifold oriented by the Liouville volume form $\Om_{n} = \tfrac{1}{n!}\Om^{n}$. For simplicity, $M$ is often assumed compact, although most claims generalize straightforwardly to the noncompact setting.

For a f\/inite-dimensional real vector space $\ste$, $S^{k}(\ste)$ and $\ext^{k}(\ste)$ denote the $k$th symmetric and antisymmetric powers of $\ste$. Purely algebraic constructions on vector spaces extend straightforwardly to vector bundles and their spaces of sections, and the same notations will be used in both contexts. For a smooth vector bundle $E \to M$, $\Ga(E)$ denotes the vector space of smooth sections of $E$, and $\ext^{k}(E)$ and $S^{k}(E)$ denote the $k$th antisymmetric and symmetric powers of $E$. By def\/inition $\ext^{0}(E) = S^{0}(E) = \cinf(M)$.

Tensors are usually indicated using the abstract index conventions (see \cite{Penrose-Rindler} or \cite{Wald}). Indices are labels and do not indicate a choice of frame. An index is in either \textit{up} or \textit{down} position. Up (down) indices label contravariant (covariant) tensors. For example, $\nabla_{i}\nabla_{j}X^{k}$ (which could also be written $(\nabla^{2}X)_{ij}\,^{k}$) indicates the second covariant derivative of the vector f\/ield $X^{k}$. Were a~frame $E_{i}$ chosen, $\nabla_{i}\nabla_{j}X^{k}$ would correspond to $(\nabla^{2}X)(E_{i}, E_{j}) = \nabla_{E_{i}}\nabla_{E_{j}}X - \nabla_{\nabla_{E_{i}}E_{j}}X$, and not to~$\nabla_{E_{i}}\nabla_{E_{j}}X$. The up index $k$ on $X^{k}$ is simply a label that indicates that $X$ is a contravariant vector f\/ield. The summation convention is used in the following form: repetition of a label in up and down position indicates the trace pairing. Enclosure of indices in square brackets (parentheses) indicates complete antisymmetrization (symmetrization) over the enclosed indices, so that, for example, $a^{ij}= a^{(ij)} + a^{[ij]}$ indicates the decomposition of a contravariant two-tensor into its symmetric and skew-symmetric parts. An index included between vertical bars $|\,|$ is omitted from an indicated (anti)symmetrization; for example $2a_{[i|jk|l]} = a_{ijkl} - a_{ljki}$.

The conventions are illustrated by the following def\/initions. The curvature $R_{ijk}\,^{l}$ and torsion~$\tau_{ij}\,^{k}$ of an af\/f\/ine connection $\nabla$ are def\/ined by $2\nabla_{[i}\nabla_{j]}X^{k} = R_{ijp}\,^{k}X^{p} - \tau_{ij}\,^{p}\nabla_{p}X^{k}$. The \textit{Ricci curvature} of $\nabla$ is $R_{ij} = R_{pij}\,^{p}$. (Sometimes, for readability, $\ric$ is written instead of $R_{ij}$.)

The antisymmetric bivector $\Om^{ij}$ inverse to $\Omega_{ij}$ is def\/ined by $\Omega^{ip}\Omega_{pj} = -\delta_{j}\,^{i}$. Indices are raised and lowered using $\Omega_{ij}$ and $\Omega^{ij}$ by contracting with these tensors consistently with the conventions $X_{i} = X^{p}\Om_{pi}$ and $X^{i} = \Om^{ip}X_{p}$. When it is raised or lowered, an index's horizontal position is maintained. The symplectic \textit{sharp} and \textit{flat} operators on a vector f\/ield $X^{i}$ and a one-form $\al_{i}$ are def\/ined by $X^{\sflat}_{i} = X^{p}\Omega_{pi}$ and $\al^{\ssharp\,i} = \Omega^{ip}\al_{p}$. These are inverses, so that $(X^{\sflat})^{\ssharp} = X$. For a vector f\/ield $X$, $X_{i}$ and $X^{\sflat}$ are synonyms and both notations will be used.

\subsection{}
Let $\Symplecto(M, \Om)$ be the group of compactly supported symplectomorphisms of $(M, \Om)$. Its Lie algebra $\symplecto(M, \Om)$ comprises compactly supported vector f\/ields $X$ that are locally Hamiltonian, meaning $\lie_{X}\Om = 0$; equivalently $X^{\sflat} = \Om(X, \,\cdot\,)$ is closed. Def\/ine the Hamiltonian vector f\/ield $\hm_{f}$ of $f \in \cinf(M)$ by $\hm_{f}^{i} = -df^{i} = \Om^{pi}df_{p}$. Since $\hm_{\{f, g\}} = [\hm_{f}, \hm_{g}]$ for the Poisson bracket $\{f, g\} = \Om(\hm_{f}, \hm_{g})$ of $f, g \in \cinf(M)$, the compactly supported Hamiltonian vector f\/ields constitute a subalgebra $\ham(M, \Om)\subset \symplecto(M, \Om)$. For a $2k$-form $\si$ and $\Om_{k} = \tfrac{1}{k!}\Om^{k}$,
\begin{gather}\label{twokformwedge}
\si \wedge \Om_{n-k} = \tfrac{1}{k!2^{k}}\Om^{i_{1}i_{2}}\vdots \Om^{i_{2k-1}i_{2k}}\si_{i_{1}\dots i_{2k}}\Om_{n} = \tfrac{1}{k!2^{k}}\si_{i_{1}}\,^{i_{1}}\,_{i_{2}}\,^{i_{2}}\,_{\dots}\,^{\dots}\,_{i_{k}}\,^{i_{k}}\Om_{n}.
\end{gather}
By~\eqref{twokformwedge}, $\{f, g\}\Om_{n} = df \wedge dg \wedge \Om_{n-1} = d(fdg\wedge \Om_{n-1})$. Since, if $M$ is compact, $\int_{M}\{f, g\} \,\Om_{n} = 0$ for $f, g \in \cinf(M)$, the subspace $\cinf_{0}(M) \subset \cinf(M)$ comprising mean zero functions is a Lie subalgebra of $\cinf(M)$, isomorphic to $\ham(M, \Om)$ via the map $f \to \hm_{f}$. If $M$ is noncompact, then $\ham(M, \Om)$ is identif\/ied with the Lie algebra $(\cinf_{c}(M), \{\dum, \dum\})$ of compactly supported smooth functions.
The subgroup $\Ham(M, \Om)\subset \Symplecto(M, \Om)_{0}$ of Hamiltonian dif\/feomorphisms of $(M, \Omega)$ comprises symplectomorphisms of $(M, \Omega)$ that can be realized as the time one f\/low of a normalized time-dependent Hamiltonian on $M \times [-1, 1]$, where \textit{normalized} means mean-zero if $M$ is compact and compactly supported if $M$ is noncompact. The inf\/initesimal generator of a f\/low by Hamiltonian dif\/feomorphisms is a Hamiltonian vector f\/ield, so that the Lie algebra of $\Ham(M, \Om)$ is $\ham(M, \Om)$.

\subsection{}
An action of a Lie group $G$ on a symplectic manifold $(M, \Om)$ is \textit{symplectic} if $G$ acts by symplectic dif\/feomorphisms. The associated Lie algebra homomorphism from the Lie algebra $\g$ of $G$ to $\symplecto(M, \Om)$, def\/ined by $x \to \X^{x}_{p} = \tfrac{d}{dt}\big|_{t = 0}\exp(-tx)\cdot p$, is \textit{weakly Hamiltonian} if there is a map $\mu\colon M \to \g^{\ast}$ such that for each $x \in \g$, the Hamiltonian vector f\/ield $\hm_{\mu(x)}$ equals the vector f\/ield $\X^{x}$ determined by the action. If~$\mu$ is equivariant with respect to the given action of~$G$ on~$M$ and the coadjoint action of $G$ on $\g^{\ast}$, then the action is \textit{Hamiltonian} and~$\mu$ is a \textit{moment map}. When $G$ is simply connected, this equivariance is equivalent to the requirement that when viewed as a map from $\g$ to $\cinf(M)$, the map $\mu$ be a Lie algebra homomorphism, $\{\mu(x), \mu(y)\} = \mu([x, y])$, and the induced action of $\g$ is said to be Hamiltonian if it satisf\/ies this last condition. For a~weakly Hamiltonian action that is not necessarily Hamiltonian, the associated $\mu$ is called a~\textit{nonequivariant moment map}. Since some authors do not require equivariance in the def\/inition of a moment map, the redundant expression \textit{equivariant moment map} is sometimes used for clarity. These def\/initions are applied in the inf\/inite-dimensional context, in which case they are to be understood formally.

\subsection{}\label{affineactionsection}
This subsection f\/ixes terminology related to af\/f\/ine actions of a Lie algebra. This is needed mainly in Section~\ref{maintheoremsection}.
A map $f\colon A \to B$ between af\/f\/ine spaces $A$ and $B$ is af\/f\/ine if there is a~linear map $\lin(f)\colon V \to W$, between the vector spaces $V$ and $W$ of translations of $A$ and $B$, such that $f(q) - f(p) = \lin(f)(q - p)$ for all $p, q \in A$. Under composition the bijective af\/f\/ine maps of $A$ to itself form the group $\Aff(A)$ of af\/f\/ine automorphisms of $A$, and $\lin\colon \Aff(A) \to {\rm GL}(V)$ is a surjective homomorphism. A one-parameter subgroup through the identity in $\Aff(A)$ has the form $\Id_{A} + t\phi + O(t^{2})$ for some af\/f\/ine map $\phi\colon A \to V$. Consequently, the Lie algebra $\aff(A)$ of $\Aff(A)$ is the vector space of af\/f\/ine maps from $A$ to $V$ equipped with the bracket $[f, g] = \lin(f)\circ g - \lin(g) \circ f$.

An \textit{affine representation} of a Lie algebra $(\g, [\dum, \dum])$ on the af\/f\/ine space $A$ is a Lie algebra homomorphism $\rho\colon \g \to \aff(A)$, that is, a linear map satisfying $\rho([a, b]) = \lin(\rho(a))\rho(b) - \lin(\rho(b))\rho(a)$ for all $a,b\in \g$. An \textit{affine action} of a Lie algebra $(\g, [\dum, \dum])$ on the af\/f\/ine space $A$ is a biaf\/f\/ine map $\pi\colon \g \times A \to A$ such that $\pi(0, p) = p$ for all $p \in A$ and $a\cdot(b\cdot p) - b\cdot(a \cdot p) = [a, b]\cdot p - p$ for all $p \in A$ and all $a, b \in \g$. That this identity is an equality of vectors in $V$ explains the need for the $p$ on its right-hand side. That $\pi$ be biaf\/f\/ine means that the maps $p \to a\cdot p = \pi(a, p)$ and $a \to a\cdot p = \pi(a, p)$ are af\/f\/ine maps from $A$ to $A$ for all $a \in \g$ and from $\g$ to $\A$ for all $p \in \A$. (Note that a biaf\/f\/ine map need not be af\/f\/ine with respect to the product af\/f\/ine structure on~$\g \times \A$.)

\begin{Lemma}\label{affineactionlemma}
For an affine space $A$ and a Lie algebra $\g$, a map $\rho\colon \g \to \aff(A)$ is an affine representation if and only if $\pi\colon \g \times A \to A$ defined by $\pi(a, p) = p + \rho(a)p$ is an affine action.
\end{Lemma}
First one shows that $\rho$ is linear if and only if $\pi$ is biaf\/f\/ine and $\pi(0, p) =0$ for all $p \in A$. Then one checks directly that, in this case, $\left((a\cdot(b\cdot p) - b\cdot (a\cdot p)\right) - ( [a, b]\cdot p - p) =(\lin(\rho(a))\rho(b) - \lin(\rho(b)\rho(a)- \rho([a, b]))p$. The representation $\rho$ is said to be \textit{associated} with the af\/f\/ine action~$\pi$. The \textit{stabilizer} of $p_{0} \in A$ under the af\/f\/ine action of $\g$ means the subalgebra $ \{a \in \g\colon a\cdot p_{0} = p_{0}\} = \{a \in \g\colon \rho(a)p_{0} = 0\}$.

A typical example is the af\/f\/ine action $X\cdot \nabla = \nabla + \lie_{X}\nabla$ of the Lie algebra of vector f\/ields on $M$ on the space of af\/f\/ine connections on $M$. The associated af\/f\/ine representation is $\rho(X) = \lie_{X}\nabla$.

A \textit{symplectic affine space} means an af\/f\/ine space $A$ equipped with a symplectic form $\Om$ inva\-riant under the action of the group of translations of $A$. An af\/f\/ine representation $\rho\colon \g \to \aff(A)$ on a symplectic af\/f\/ine space $(A, \Om)$ is \textit{symplectic affine} if each $\rho(a)$ is symplectic, meaning the linear part $\lin(\rho(a))$ satisf\/ies $\Om(\lin(\rho(a))u, v) + \Om(u, \lin(\rho(a))v) = 0$ for all $a \in \g$ and $u, v \in V$. Similarly, an af\/f\/ine action is symplectic af\/f\/ine if the corresponding af\/f\/ine representation is symplectic.

\subsection{}
An af\/f\/ine connection is \textit{symplectic} if $\nabla_{i}\Om_{jk} = 0$. For background about symplectic connections see \cite{Bieliavsky-Cahen-Gutt-Rawnsley-Schwachhofer} or \cite{Fedosov} and the references therein. The space $\tcon(M, \Om)$ of symplectic af\/f\/ine connections on $(M, \Om)$ is an af\/f\/ine space modeled on the vector space $\Ga(\ctm \tensor S^{2}(\ctm))$ of covariant tensors $\Pi_{ijk} = \Pi_{i(jk)}$ symmetric in the last two entries. Precisely, if $\nabla \in \tcon(M, \Om)$ and $\bnabla = \nabla + \Pi_{ij}\,^{k}$ then $\bnabla_{i}\Om_{jk} = -2\Pi_{i[jk]}$, so that $\bnabla \in \tcon(M, \Om)$ if and only if $\Pi_{i[jk]} = 0$. Note that $\Pi_{ip}\,^{p} =0$, ref\/lecting that a symplectic af\/f\/ine connection preserves the associated volume form $\Om_{n}$. The af\/f\/ine subspace $\symcon(M, \Om)\subset \tcon(M, \Om)$ comprising torsion-free symplectic af\/f\/ine connections is modeled on the vector space $\ssec^{3}(M)$, for the dif\/ference tensor of torsion-free connections satisf\/ies $\Pi_{[ij]}\,^{k} = 0$, and with $\Pi_{i[jk]} = 0$, this implies $\Pi_{ijk} = \Pi_{(ijk)}$. The space $\symcon(M, \Om)$ is nonempty, for if $\bnabla$ is any torsion-free af\/f\/ine connection then $\nabla = \bnabla + \tfrac{2}{3}\Om^{kp}\bnabla_{(i}\Om_{j)p}$ is symplectic. It is convenient to write $T_{\nabla}\tcon(M, \Om) = \Ga(\ctm \tensor S^{2}(\ctm))$ or $T_{\nabla}\symcon(M, \Om)= \ssec^{3}(M)$ for the tangent space to the af\/f\/ine spaces $\tcon(M, \Om)$ or $\symcon(M, \Om)$ at $\nabla$. The focus here is on $\symcon(M, \Om)$, and it is mainly in Section~\ref{gaugeactionsection} that reference will be made to $\tcon(M, \Om)$. Throughout this paper \textit{\textbf{symplectic connection} means a torsion-free symplectic connection, while \textbf{symplectic affine connection} means a connection preserving $\Om$ but possibly having torsion}.

Since, for $\nabla \in \symcon(M, \Om)$, $\Om_{n}$ is $\nabla$-parallel, the curvature $-R_{ijp}\,^{p}$ of the connection induced on $\ext^{2n}(\ctm)$ by $\nabla$ vanishes. Hence $\nabla$ has symmetric Ricci tensor, for $2R_{[ij]} = - R_{ijp}\,^{p} = 0$. By the Ricci identity, $0 = 2\nabla_{[i}\nabla_{j]}\Om_{kl} = -2R_{ij[kl]}$, where $R_{ijkl} = R_{ijk}\,^{p}\Om_{pl}$. The algebraic Bianchi identity and $R_{ij[kl]} = 0$ yield $R_{p}\,^{p}\,_{ij} = -2R_{ip}\,^{p}\,_{j} = 2R_{pij}\,^{p} = 2R_{ij}$, from which it follows that every nontrivial trace of $R_{ijkl}$ is a constant multiple of $R_{ij}$.

\subsection{}
Let $(\ste, \Om)$ be a symplectic vector space. On $\Wl^{p, q}(\ste) = \ext^{p}(\ste)\tensor S^{q}(\ste)$ def\/ine
\begin{gather}\label{ompairing}
\Om( \al, \be) = (-1)^{p}\tfrac{1}{p!}\al_{i_{1}\dots i_{p}j_{1}\dots j_{q}}\be^{i_{1}\dots i_{p}j_{1}\dots j_{q}}.
\end{gather}
This pairing is graded symmetric in the sense that $\Om( \al, \be)= (-1)^{|\al||\be|}\Om( \be, \al)$ for homogeneous elements $\al$ and $\be$, where $|\al| = p + q$ is the \textit{degree} of $\al \in \Wl^{p,q}(\ste)$. Integrating~\eqref{ompairing} yields a~$\Symplecto(M, \Om)$-invariant graded symmetric pairing on $\Ga(\Wl^{p, q}(\ctm))$ def\/ined by
\begin{gather}\label{wpairing}
\laa \al, \be \raa = \int_{M}\Om(\al, \be)\,\Om_{n} = (-1)^{p}\tfrac{1}{p!}\int_{M}\al_{i_{1}\dots i_{p}j_{1}\dots j_{q}}\be^{i_{1}\dots i_{p}j_{1}\dots j_{q}}\,\Om_{n}.
\end{gather}
Using an almost complex structure compatible with $\Om$ it is straightforward to show that the pairing $\laa \al, \be \raa$ is weakly nondegenerate in the sense that if $ \laa \al, \be \raa = 0$ for all compactly supported $\be \in \Ga(\Wl^{p, q}(\ctm))$ then $\al = 0$.

\subsection{}
Fix a $2n$-dimensional symplectic vector space $(\ste, \Om)$. The \textit{(linear) symplectic frame bundle} $\pi\colon \sfr \to M$ is the $G = Sp(\ste, \Om)$ principal bundle whose f\/iber over $p \in M$ comprises symplectic linear maps $u\colon (\ste, \Om) \to (T_{p}M, \Om)$. A symplectic af\/f\/ine connection $\nabla$ determines and is determined by a principal connection on~$\sfr$. The space $\tcon(M, \Om)$ of symplectic af\/f\/ine connections on~$M$ is an af\/f\/ine space modeled on the vector space of $1$-forms taking values in the adjoint bundle $\Ad(\sfr) = \sfr \times_{\Ad}\g$ associated with $\sfr$ by the adjoint action of~$G$ on its Lie algebra $\g = \sp(\ste, \Om)$.

\subsection{}
Given $\al, \be \in T_{\nabla}\tcon(M, \Om)$, composing the bundle-valued two-form $(\al \wedge \be)_{ijk}\,^{l} = 2\al_{[i|p|}\,^{l}\be_{j]k}\,^{p}$ $ = -2\al_{[i}\,^{pl}\be_{j]kp}$ with the trace on endomorphisms yields the two-form $\tr (\al \wedge \be)_{ij} = 2\al_{[i|p|}\,^{q}\be_{j]q}\,^{p}$. By~\eqref{twokformwedge}, $\tr(\al \wedge \be)\wedge \Om_{n-1} = -\al_{ijk}\be^{ijk} \Om_{n}$. By~\eqref{wpairing} the two-form $\sOm$ def\/ined on $\tcon(M, \Om)$ by
\begin{gather*}
\sOm_{\nabla}(\al, \be) = -\int_{M} \tr(\al \wedge \be) \wedge \Om_{n-1} = \int_{M} \al_{ijk}\be^{ijk} \, \Om_{n} = \laa \al, \be \raa, \qquad \al, \be \in T_{\nabla}\tcon,
\end{gather*}
is weakly nondegenerate. The translation invariance, $\sOm_{\nabla + \Pi} = \sOm_{\nabla}$, for $\Pi \in T_{\nabla}\tcon(M, \Om)$ means~$\sOm$ is parallel, so closed, and so~$\sOm$ is a symplectic form. The subspace $\symcon(M, \Om)\subset \tcon(M, \Om)$ is symplectic.

The group of dif\/feomorphisms acts on the space of af\/f\/ine connections on the right, by pullback. Precisely $\phi^{\ast}(\nabla)$ is def\/ined by $\phi^{\ast}(\nabla)_{X}Y = T\phi^{-1}(\nabla_{T\phi(X)}T\phi(Y))$. This action preserves the torsion-free connections. Consequently, $\Symplecto(M, \Om)$ acts on $\symcon(M, \Om)$ and $\tcon(M, \Om)$ by pullback, and this action preserves $\sOm$. Integration by parts yields the corresponding inf\/initesimal statement, $\sOm_{\nabla} (\lie_{X}\al, \be) + \sOm_{\nabla}( \al, \lie_{X}\be ) = 0$, for $X \in \symplecto(M, \Om)$.

\section{Symmetric tensors on symplectic manifolds}\label{symmetrictensorsection}

\subsection{}
The symmetric tensor algebra $\mathcal{S}(\ste)$, which is, by def\/inition, a quotient of the tensor algebra, can be identif\/ied, as a graded vector space, with the graded vector space $\oplus_{k \geq 0}S^{k}(\ste)$ equipped with the symmetric product $(\al \sprod \be)^{i_{1}\dots i_{k+l}} = \al^{(i_{1}\dots i_{k}}\be^{i_{k+1}\dots i_{k+l})}$, for $\al \in S^{k}(\ste)$ and $\be \in S^{l}(\ste)$. The subspace $S^{k}(\ste)$ is identif\/ied with the space $\pol^{k}(\std)$ of homogeneous degree $k$ polynomials on the dual vector space $\std$, and the product $\sprod$ corresponds to multiplication of polynomials. Elements of the formally completed symmetric tensor algebra $\hat{\mathcal{S}}(\ste) = \prod_{k \geq 0}S^{k}(\ste)$ are identif\/ied with formal inf\/inite sums of symmetric tensors, and correspond with formal power series on~$\std$.

Let $\mathcal{S}(\ctm) = \oplus_{k \geq 0}S^{k}(\ctm)$ and $\hat{\mathcal{S}}(\ctm) = \prod_{k\geq 0}S^{k}(\ctm)$. Set $\ssec^{k}(M) = \Ga(S^{k}(\ctm))$. It is more convenient to work with the direct product $\hssec(M)= \prod_{k \geq 0}\!\ssec^{k}(M)$ than with $\Ga(\mathcal{S}(\ctm))$ or the direct sum $\ssec(M)= \oplus_{k \geq 0}\ssec^{k}(M)$. Because a section of $\mathcal{S}(\ctm)$ need not have bounded degree, the space $\Ga(\mathcal{S}(\ctm))$ can be larger than $\ssec(M)$. On the other hand, the direct product~$\hssec(M)$ equals the space of sections of $\hS(\ctm)$. An element of $\hssec(M)$ is regarded as a formal inf\/inite sum of symmetric tensors on~$M$. There are inclusions, $\ssec(M) \subset \Ga(\mathcal{S}(M)) \subset \hssec(M)$, as associative algebras. In general, the algebraic constructions considered here are valid on any of these spaces, and, although they are usually formulated in terms of~$\hssec(M)$, their reformulations for the smaller spaces are generally trivial or straightforward.

An element $\al \in S^{k}(\ste)$ or $\al \in \ssec^{k}(M)$ is said to have degree $|\al| = k$. For all $i \geq 1$, the subspaces $\mathcal{S}_{k}(\ste) = \oplus_{i \geq k}S^{i}(\ste)$ and $\hat{\mathcal{S}}_{k}(\ste) = \prod_{i \geq k}S^{i}(\ste)$ are ideals with respect to the commutative algebra structure determined by $\sprod$.
The $p = 0$ case of~\eqref{ompairing} extends to a graded symmetric pairing on~$\hS(\ste)$ in such a way that $S^{k}(\ste)$ and $S^{l}(\ste)$ are orthogonal if $k \neq l$. Likewise, as commutative algebras, $\ssec(M)$ and $\hssec(M)$ are f\/iltered by the ideals $\ssec_{k}(M)$ and $\hssec_{k}(M)$ comprising sums of tensors of degree at least $k$, and the pairing~\eqref{wpairing} extends to a graded symmetric pairing on $\hssec(M)$.

\subsection{}\label{algebraicbracketsection}
If $(\ste, \Om)$ is a symplectic vector space, then $\mathcal{S}(\std)$ carries a Lie bracket, determining with~$\sprod$ a~Poisson algebra, that can be def\/ined most simply as the unique Poisson bracket $(\dum, \dum)$ on $(\mathcal{S}(\std), \sprod)$ such that $(\al, \be) = \Om^{ij}\al_{i}\be_{j}$ for all $\al, \be \in \std$. By def\/inition, $(\dum, \dum)$ has degree $-2$. The bracket $(\dum, \dum)$ can be described explicitly as follows. The space of functions on $\ste$ carries the Poisson structure determined by $\Om$. Transporting this Poisson structure to $\mathcal{S}(\std)$ via the identif\/ication of $\mathcal{S}(\std)$ with the ring of polynomials on $\ste$, yields the bracket $(\dum, \dum)$ on $\mathcal{S}(\std)$. Concretely, for $\al \in S^{k}(\std)$, the corresponding polynomial $\hat{\al} \in \pol^{k}(\ste)$ is $\hat{\al}(u) = \al_{i_{1}\dots i_{k}}u^{i_{1}}\cdots u^{i_{k}}$, and it generates the Hamiltonian vector f\/ield $\hm_{\hat{\al}}^{i} = -k\al^{i}\,_{i_{1}\dots i_{k-1}}u^{i_{1}}\cdots u^{i_{k-1}}$ on $\ste$. The Poisson bracket of $\hat{\al}$ and $\hat{\be}$ is the element $\Om(\hm_{\hat{\al}}, \hm_{\hat{\be}})$ of $\pol^{k+l-2}(\ste)$ corresponding to $(\al, \be) \in S^{k+l-2}(\std)$ def\/ined by
\begin{align}
(\al, \be)_{i_{1}\dots i_{k+l-2}} &= kl\al_{p(i_{1}\dots i_{k-1}}\be_{i_{k} \dots i_{k+l-2})}\,^{p}\label{algebraicbracket}\\
& = \tfrac{kl}{k+l-2}\left((k-1)\al_{pi_{1}(i_{2}\dots i_{k-1}}\be_{i_{k}\dots i_{k+l-2})}\,^{p} - (l-1)\be_{pi_{1}(i_{2}\dots i_{l-1}}\al_{i_{l}\dots i_{k+l-2})}\,^{p}\right).\nonumber
\end{align}
In the case $l = 1$ and $k > 1$, $(\al, \be) = k\be^{p}\al_{pi_{1}\dots i_{k-1}}$ is simply $k$ times interior multiplication of the vector f\/ield $\be^{i}$ in $\al$. The Poisson bracket~\eqref{algebraicbracket} extends to $\hat{\mathcal{S}}(\std)$ with the same def\/inition.

Although it is true by construction that $(\dum, \dum)$ is a Lie bracket that forms with the symmetric product $\sprod$ a Poisson structure, this can also be checked directly. Let $\ga \in S^{m}(\std)$. That $(\dum, \dum)$ is a Lie bracket follows by summing cyclic permutations of the identity
\begin{gather*}
((\al, \be), \ga)_{i_{1}\dots i_{k+l+m-4}} = klm\left((k-1)\al_{pq(i_{1}\dots i_{k-2}}\be_{i_{k-1}\dots i_{k+l-3}}\,^{p}\ga_{i_{k+l-2}\dots i_{k+l+m-4})}\,^{q} \right.\\
\left. \hphantom{((\al, \be), \ga)_{i_{1}\dots i_{k+l+m-4}} =}{}- (l-1)\be_{pq(i_{1}\dots i_{l-2}}\al_{i_{k-1}\dots i_{k+l-3}}\,^{p}\ga_{i_{k+l-2}\dots i_{k+l+m-4})}\,^{q}\right)
\end{gather*}
(obtained by using both equalities in~\eqref{algebraicbracket}). Similarly, verifying $(\al \sprod \be, \ga) = \al \sprod (\be, \ga) + (\al, \ga)\sprod \be$ is straightforward using~\eqref{algebraicbracket}, and so $(\dum, \dum)$ is a Poisson bracket.

Note that $(S^{2}(\std), (\dum, \dum))$ is a Lie subalgebra of $(\mathcal{S}(\std), (\dum, \dum))$. Using the easily checked invariance $\Om((\al, \be), \ga) + \Om(\be, (\al, \ga)) = 0$ for $\al \in S^{2}(\std)$ and $\be, \ga \in S^{k}(\std)$, it is straightforward to show that the map $S^{2}(\std) \to \eno(\ste)$ def\/ined by $\al \to (\al, \dum)\big|_{S^{1}(\std)}$ is a Lie algebra isomorphism onto the symplectic Lie algebra $\sp(\std, \Om)$ (see Section~2 of~\cite{Kostant-weyl} for a detailed discussion).

\begin{Remark}\label{weylalgebraremark}
The Weyl algebra of the symplectic vector space $(\ste, \Om)$ is the quotient of the tensor algebra of $\ste$ by the ideal generated by elements of the form $u \tensor v - v \tensor u - \Om(u, v)$ for $u, v \in \ste$. The additive f\/iltration of the Weyl algebra at level $k$ is spanned by $k$-fold products of elements of $\ste$. Its associated graded algebra is isomorphic as a vector space to the symmetric algebra of $\ste$ and carries a Poisson bracket def\/ined by projecting the commutator of elements of level $k$ and $l$ onto level $k + l -2$. Applying this construction to $(\std, \Om)$ yields the algebraic bracket~\eqref{algebraicbracket}.
\end{Remark}

\begin{Remark}\label{prolongationremark}
Another construction of~\eqref{algebraicbracket} can be given in terms of the notion of prolongation. This is recalled following \cite{Guillemin}; see also \cite[Chapter~7.3]{Sternberg-differentialgeometry} and \cite[Chapter~I.1]{Kobayashi-transformationgroups}.

The f\/irst prolongation $\g^{(1)}$ of the subspace $\g \subset \hom(\ste, \stw)\simeq \stw \tensor \std$ is def\/ined by $\g^{(1)} = \{A \in \hom(\ste, \g) \simeq \g \tensor \std\colon A(u)v = A(v)u \,\,\text{for all}\,\, u, v \in \ste\}$, and, for $k \geq 2$, the $k$th prolongation~$\g^{(k)}$ of~$\g$ is def\/ined inductively by $\g^{(k)} = (\g^{(k-1)})^{(1)}$. As vector spaces,
\begin{gather}\label{prolong1}
\g^{(k)} = \big(\g \tensor \tensor^{k}(\std)\big) \cap \big(\stw \tensor S^{k+1}(\std)\big).
\end{gather}
When $\stw = \ste$, so $\g \subset \eno(\ste)$, def\/ine $\g^{(0)} = \g$ and $\g^{(-1)} = \ste$. Using the identif\/ication~\eqref{prolong1}, a~graded Lie bracket is def\/ined on the direct sum $\oplus_{k \geq -1}\g^{(k)}$ and its formal completion $\prod_{k \geq -1}\g^{(k)}$ as follows. Def\/ine an antisymmetric bilinear pairing $[\dum, \dum]\colon \ste \tensor S^{k+1}(\std) \times \ste \tensor S^{l+1}(\std) \to \ste \tensor S^{k+l+1}(\std)$ by setting
\begin{gather}\label{prolonglie}
[v_{1}\tensor s_{1}, v_{2}\tensor s_{2}] =v_{1} \tensor ((D_{v_{2}}s_{1})\sprod s_{2} ) - v_{2} \tensor (s_{1}\sprod (D_{v_{1}}s_{2})),
\end{gather}
for $v_{1}, v_{2} \in \ste$, $s_{1} \in S^{k}(\std)$, and $s_{2} \in S^{l}(\std)$, and extending linearly, where $D_{v}\colon S^{k}(\std) \to S^{k-1}(\std)$ is def\/ined by $D_{v}s = k\imt(v)s$ for $v \in \ste$ (the constant factor is chosen so that $D_{v}$ is a derivation of the symmetric product~$\sprod$). The bracket~\eqref{prolonglie} restricts to yield a bracket $[\dum, \dum]\colon \g^{(k)} \times \g^{(l)} \to \ste \tensor S^{k+l+1}(\std)$, and it is straightforward to check that the image is contained in $\g^{(k+l)}$ and that the resulting bracket satisf\/ies the Jacobi identity, so makes $\oplus_{k \geq -1}\g^{(k)}$ and $\prod_{k \geq -1}\g^{(k)}$ into Lie algebras. In the case $(\ste, \Om)$ is a symplectic vector space and $\g = \sp(\ste, \Om)$, this can be seen as follows. If the dif\/ferential of the f\/low of a vector f\/ield on $\ste$ preserves the symplectic frame bundle $\ste \times Sp(\ste, \Om)$ then the $(k+1)$st coef\/f\/icient of the Taylor expansion of the vector f\/ield takes values in $\g^{(k)}$. The Lie bracket is that induced by taking jets of Lie brackets of vector f\/ields. Since $\stw = \ste$, an element of $\g^{(0)}$ can be viewed as an endomorphism of $\ste$; the def\/inition~\eqref{prolonglie} is made so that the bracket $[\dum, \dum]$ agrees on $\g^{(0)}$ with the algebraic commutator of endomorphisms. That is $[A, B]_{i}\,^{j} = (A\circ B - B \circ A)_{i}\,^{j} = A_{p}\,^{j}B_{i}\,^{p} - B_{p}\,^{j}A_{i}\,^{p}$ for $A, B \in \g^{(0)} \subset \hom(\ste, \ste)$. An endomorphism $A_{i}\,^{j}$ of $(\ste, \Om)$ is inf\/initesimally symplectic if and only if $A_{[ij]} = 0$, so, using $\Om$, $\g$ is identif\/ied with $S^{2}(\std)$ equipped with the Lie bracket $[\al, \be]_{ij} = 2\al_{p(i}\be_{j)}\,^{p} = 2(\al, \be)_{ij}$ for $\al, \be \in S^{2}(\std)$. Likewise, using $\Om$ to identify the leading $\ste$ ($=\stw$) in~\eqref{prolong1} with $\std$, it follows from~\eqref{prolong1} that the $k$th prolongation $\g^{(k)}$ is identif\/ied with $S^{k+2}(V^{\ast})$. Then $\oplus_{k \geq -1}\g^{(k)}$ and $\prod_{k \geq -1}\g^{(k)}$ are identif\/ied with the ideals $\mathcal{S}_{1}(\std) \subset \mathcal{S}(\std)$ and $\hat{\mathcal{S}}_{1}(\std) \subset \hat{\mathcal{S}}(\std)$ and~\eqref{prolonglie} induces on these an algebraic Lie bracket $[\dum, \dum]$ of degree $-2$ with respect to the grading of $\mathcal{S}(\ste)$. The bracket $[\dum, \dum]$ is related to the bracket~\eqref{algebraicbracket} by $kl[\al,\be] = (k+l-2)(\al, \be)$ for $\al \in S^{k}(\std)$ and $\be \in S^{l}(\std)$.

The preceding can be recast in the following manner. The graded linear map $\eta\colon \mathcal{S}(\std) \to \mathcal{S}(\std)$ def\/ined by $\eta(\al) = k\al$ for $\al \in S^{k}(\std)$ restricts to a linear isomorphism on $\mathcal{S}_{1}(\std)$. Then
\begin{gather*}
\{0\} \longrightarrow \rea \stackrel{\imt}{\longrightarrow} (\mathcal{S}(\std), (\dum, \dum))\stackrel{\eta}{\longrightarrow}(\mathcal{S}_{1}(\std), [\dum, \dum]) \longrightarrow \{0\}
\end{gather*}
is a central extension of Lie algebras, where $\imt$ is the inclusion of $\rea$ as the central ideal $S^{0}(\std) \subset (\mathcal{S}(\std), (\dum, \dum))$. The map $\eta$ is an algebraic version of the operator associating to the polynomial function $\hat{\al}(u) = \al_{i_{1}\dots i_{k}}u^{i_{1}}\cdots u^{i_{k}}$ the Hamiltonian vector f\/ield $\hm_{\hat{\al}}^{i} = -k\al^{i}\,_{i_{1}\dots i_{k-1}}u^{i_{1}}\cdots u^{i_{k-1}}$. Its restriction to $\mathcal{S}_{1}(\std)$ is an isomorphism between $(\dum, \dum)$ and $[\dum, \dum]$.
\end{Remark}

\begin{Remark}\label{framebundleremark}
A horizontal distribution on a symplectic f\/ibration such that the corresponding parallel transport preserves the f\/iberwise symplectic structure is also called a \textit{symplectic connection} (see for example~\cite{Mcduff-Salamon}). The symplectic af\/f\/ine connections considered here are symplectic connections in this more general sense, although the converse is clearly false. For this reason (and also to honor their role in Fedosov's deformation quantization) what are here called \textit{symplectic connections} are sometimes called \textit{Fedosov connections}; here the notions are distinguished terminologically as \textit{linear} and \textit{nonlinear} symplectic connections.

The two notions encapsulate dif\/ferent points of view. Each is equivalent to the data of a~principal connection on what might be called the symplectic frame bundle of $M$, the point being that \textit{symplectic frame bundle} has two possible meanings, depending on whether \textit{symplectic} refers to the symplectic linear group or the symplectomorphism group of the reference symplectic vector space $(\ste, \Om)$. The linear symplectic frame bundle is the reduction of the usual linear frame bundle determined by restricting to symplectic frames. A much f\/labbier notion is obtained by instead considering as frames at $p \in M$ all symplectomorphisms from $(\ste, \Om)$ to $(T_{p}M, \Om)$ mapping $0 \in \ste$ to $0 \in T_{p}M$. Although to def\/ine such a nonlinear symplectic frame bundle rigorously requires restricting the class of symplectomorphisms considered, or working formally, e.g., with inf\/inite jets, morally there results a principal bundle with structure group an inf\/inite-dimensional group of (perhaps formal) symplectomorphisms of $(\ste, \Om)$, and, by construction, $\sfr$ is a reduction of this bundle. Correspondingly, there are two quite distinct notions of symmetries of $\symcon(M, \Om)$. The usual one is to consider principal bundle automorphisms of the linear symplectic frame bundle $\sfr$, regarded as a principal bundle for the f\/inite-dimensional group $G = \operatorname{Sp}(\ste, \Om)$ of linear symplectic transformations. On the other hand, the symplectic frame bundle~$\sfr$ can be regarded as a subbundle of a principal bundle with structure group some inf\/inite-dimensional Lie group~$\G$ of symplectomorphisms of~$(\ste, \Om)$ f\/ixing the origin.

The automorphism group of $\sfr$ comprises the $G$-equivariant bundle maps from $\sfr$ to itself. The \textit{group $\gge(\sfr)$ of gauge transformations} of~$\sfr$ is its subgroup comprising bundle automorphisms of~$\sfr$ covering the identity. A gauge transformation is naturally identif\/ied with a~map from~$\sfr$ to~$G$, equivariant with respect to the action of~$G$ on itself by conjugation, or, equivalently, with a~section of the bundle $\sfr \times_{G}G$ associated with~$\sfr$ by this action. Via this identif\/ication the group structure corresponds to f\/iberwise composition. An \textit{infinitesimal gauge transformation} is a~section of the adjoint bundle~$\Ad(\sfr)$. These form a Lie algebra~$\lge(\sfr)$ under f\/iberwise Lie bracket.

In the case of a frame bundle, a (inf\/initesimal) gauge transformation can be identif\/ied with a~section of the bundle $\eno(TM)$ of endomorphisms of the tangent bundle. Precisely, an element of~$\gge(\sfr)$ is identif\/ied with $g_{i}\,^{j} \in \Ga(\eno(TM))$ such that $g_{iq}g_{j}\,^{q} = g_{i}\,^{p}g_{j}\,^{q}\Om_{pq} = \Om_{ij}$. The inverse section $g^{-1}$ then satisf\/ies $(g^{-1})_{ij} = -g_{ji}$. Similarly, an inf\/initesimal gauge transformation of $\sfr$ is identif\/ied with a section $\al_{i}\,^{j} \in \Ga(\eno(TM))$ such that $\al_{[ij]} = 0$; in general it is more convenient to identify $\al$ with the corresponding element $\al_{ij}$ of $\ssec^{2}(M)$. Alternatively, the Lie algebra $\g$ is isomorphic as a $G$-module to $S^{2}(\std)$ via the map $A \to A^{\sflat}$ def\/ined by $A^{\sflat}(x, y) = \Om(Ax, y)$, and via this isomorphism $\Ga(\Ad(\sfr))$ is identif\/ied with $(\ssec^{2}(M), [\dum, \dum])$.

The preceding all makes formal sense with the nonlinear symplectic frame bundle in place of $\sfr$, and with the corresponding structure group $\G$ in place of $G$. Such notions can be given rigorous sense at the Lie algebra level. The formal analogue of the inf\/initesimal $\G$-gauge transformations is the Lie algebra $(\hssec_{2}(M), [\dum, \dum])$, corresponding to the f\/iberwise action of the truncated prolongation $(\prod_{i \geq 0}\g^{(i)}, [\dum, \dum])$ (as in Remark~\ref{prolongationremark}), comprising formal sums of covariant symmetric tensors of rank at least two.
\end{Remark}

\subsection{}
The pairing~\eqref{wpairing} on $\ssec^{k}(M)$ extends to a $\Symplecto(M, \Om)$-invariant bilinear pairing on $\hssec(M)$ by declaring $\ssec^{k}(M)$ and $\ssec^{l}(M)$ orthogonal if $k \neq l$. This pairing is graded symmetric in the sense that $\laa \al, \be \raa = (-1)^{|\al||\be|}\laa \be, \al \raa$ for homogeneous elements $\al$ and $\be$. If a function is regarded as a $0$-tensor, then $\laa \dum, \dum \raa$ agrees with the $L^{2}$ inner product $\laa f, g \raa = \int_{M}fg \,\Om$ on $\cinf(M)$.

A linear operator $\dop\colon \ssec^{p}(M) \to \ssec^{q}(M)$ has \textit{degree} $|\dop| = q-p$. More generally, a linear operator $\dop\colon \ssec(M) \to \ssec(M)$ has degree $|\dop| = p$ if $\P(\al) \in \ssec^{k+p}(M)$ whenever $\al \in \ssec^{k}(M)$. The \textit{$($formal$)$ adjoint} $\dop^{\ast}\colon \ssec^{q}(M) \to \ssec^{p}(M)$ of $\dop$ is the degree $-|\dop|$ operator def\/ined by $\laa \dop \al, \be\raa = (-1)^{|\al||\dop|}\laa \al, \dop^{\ast}\be\raa$. The choice of sign guarantees that $(\dop^{\ast})^{\ast} = \dop$ and $(\P\Q)^{\ast} = (-1)^{|\P||\Q|}\Q^{\ast}\P^{\ast}$.

For example, for $\al \in \ssec^{2}(M)$, $(\al, \dum)^{\ast} = -(\al, \dum)$, because the pairing $\laa\dum, \dum\raa$ is invariant with respect to the linear map $(\al, \dum)\colon \ssec(M) \to \ssec(M)$ in the sense that
\begin{gather}\label{algliepairing}
\laa (\al, \Pi_{1}), \Pi_{2} \raa + \laa \Pi_{1}, (\al, \Pi_{2})\raa = 0,\qquad \text{for}\quad \Pi_{1}, \Pi_{2} \in \ssec^{k}(M).
\end{gather}

\subsection{}
For $\nabla \in \symcon(M, \Om)$ and $\al\in \ssec^{k}(M)$ def\/ine $\dn\al_{i_{1}\dots i_{k+1}}$ by $\dn\al_{i_{1}\dots i_{k+1}}= 2\nabla_{[i_{1}}\al_{i_{2}]i_{3}\dots i_{k+1}}$, and def\/ine the degree $-1$ \textit{symplectic divergence} operator $\sd_{\nabla}\colon \hssec(M) \to \hssec(M)$ by
\begin{gather*}
\sd_{\nabla} \al_{i_{1}\dots i_{k-1}} = (-1)^{k-1}\tfrac{1}{2}(\dn\al)_{p}\,^{p}\,_{i_{1}\dots i_{k-1}}
= (-1)^{k-1}\nabla_{p}\al_{i_{1}\dots i_{k-1}}\,^{p}, \qquad \al \in \ssec^{k}(M).
\end{gather*}
The formal adjoint $\sd_{\nabla}^{\ast}\colon \ssec^{k}(M) \to \ssec^{k+1}(M)$ of $\sd_{\nabla}$ is given by
\begin{gather*}
\sd_{\nabla}^{\ast}\al_{i_{1}\dots i_{k+1}} =- \nabla_{(i_{1}}\al_{i_{2}\dots i_{k+1})} = - \nabla_{i_{1}}\al_{i_{2}\dots i_{k+1}} + \tfrac{k}{k+1}\dn\al_{i_{1}(i_{2}\dots i_{k+1})},\qquad \al \in \ssec^{k}(M).
\end{gather*}
When it is not necessary to indicate the dependence on $\nabla$ the subscripts are omitted and there are written $\sd$ and $\sd^{\ast}$ instead of $\sd_{\nabla}$ and $\sd^{\ast}_{\nabla}$. When $\sd^{\ast}$ is applied to functions the subscript is always omitted, because $\sd^{\ast}f = -df = \hm_{f}^{\sflat}$. Note that the Poisson bracket $\{f, g\}$ of $f, g, \in \cinf(M)$ is expressible as $\{f, g\}= (\sd^{\ast}f, \sd^{\ast}g)$.

Lemma~\ref{sdcommutationlemma} is needed in the proof of Lemma~\ref{poissonlemma}.
\begin{Lemma}\label{sdcommutationlemma} For $\nabla \in \symcon(M, \Om)$ and $\al \in \ssec^{k}(M)$,
\begin{align}
(k+1)\sd\sd^{\ast} \al_{i_{1}\dots i_{k}} + k\sd^{\ast}\sd\al_{i_{1} \dots i_{k}} &= (-1)^{k}\left(2kR_{(i_{1}}\,^{p}\al_{i_{2}\dots i_{k})p} - k(k-1)R^{p}\,_{(i_{1}i_{2}}\,^{q}\al_{i_{3}\dots i_{k})pq}\right)\nonumber\\
& = (-1)^{k}\left((\al, \ric)_{i_{1}\dots i_{k}} - k(k-1)R^{p}\,_{(i_{1}i_{2}}\,^{q}\al_{i_{3}\dots i_{k})pq}\right).\label{sdsdast}
\end{align}
\end{Lemma}

\begin{proof}
The Ricci identity yields
\begin{align}
(-1)^{k-1}\sd^{\ast}\sd\al_{i_{1}\dots i_{k}} & = \nabla_{(i_{1}}\nabla^{p}\al_{i_{2}\dots i_{k})p} \nonumber\\
&= \nabla^{p}\nabla_{(i_{1}}\al_{i_{2}\dots i_{k})p} + (k-1)R^{p}\,_{(i_{1}i_{2}}\,^{q}\al_{i_{3}\dots i_{k})pq} - R^{p}\,_{(i_{1}}\al_{i_{2}\dots i_{k})p}.
\label{presd}
\end{align}
Combining~\eqref{presd} with
\begin{align*}
(-1)^{k}(k+1)\sd \sd^{\ast}\al_{i_{1}\dots i_{k}}& = (k+1)\nabla^{p}\nabla_{(i_{1}}\al_{i_{2}\dots i_{k}p)}= k\nabla^{p}\nabla_{(i_{1}}\al_{i_{2}\dots i_{k})p} + \nabla^{p}\nabla_{p}\al_{i_{1}\dots i_{k}} \\
& = k\nabla^{p}\nabla_{(i_{1}}\al_{i_{2}\dots i_{k})p} + k R^{p}\,_{(i_{1}}\al_{i_{2}\dots i_{k})p},
\end{align*}
yields~\eqref{sdsdast}.
\end{proof}

\subsection{}
The Schouten bracket of contravariant symmetric tensors on a smooth manifold $M$ is def\/ined by regarding such tensors as functions on the cotangent bundle $\ctm$ polynomial in the f\/ibers and pairing them via the canonical Poisson bracket. When $M$ is symplectic the resulting Poisson algebra structure can be transferred via the symplectic form to the algebra of covariant symmetric tensors. The resulting degree $-1$ pairing $\shl\dum,\dum\shr \colon \ssec^{k}(M) \times \ssec^{l}(M) \to \ssec^{k+l-1}(M)$ is expressible in terms of any symplectic connection $\nabla \in \symcon(M, \Om)$ by
\begin{gather}\label{schouten}
\shl\al, \be\shr_{i_{1}\dots i_{k+l-1}} = -k\al_{p(i_{1}\dots i_{k-1}}\nabla^{p}\be_{i_{k}\dots i_{k+l-1})} + l\be_{p(i_{1}\dots i_{l-1}}\nabla^{p}\al_{i_{l}\dots i_{k+l-1})}.
\end{gather}
By the def\/inition~\eqref{schouten}, the Schouten bracket of functions is trivial. In~\eqref{schouten} the signs are chosen so that $\shl X^{\sflat}, Y^{\sflat}\shr = [X, Y]^{\sflat}$ for $X, Y \in \Ga(TM)$. More generally, if $X \in \symplecto(M, \Om)$ then $\shl X^{\sflat}, \al\shr = \lie_{X}\al$. In particular, $\lie_{\hm_{f}}\al = \shl\al, df\shr = \shl \sd^{\ast}f, \al \shr$ for $f \in \cinf(M)$. Since $\symplecto(M, \Om)$ is a Lie algebra, the Schouten bracket of closed one-forms is a closed one-form.
The operator $\sd^{\ast}$ is a homomorphism from $(\cinf(M), \{\dum, \dum\})$ to $(\ham(M,\Om), \shl\dum, \dum\shr)$, as
\begin{gather}\label{sdschouten}
\shl \sd^{\ast}f, \sd^{\ast}g\shr = \shl df, dg\shr = - d\{f, g\} = \sd^{\ast}\{f, g\},\qquad \text{for}\quad f, g \in \cinf(M).
\end{gather}
Since the pairing~\eqref{schouten} does not depend on the choice of $\nabla \in \symcon(M, \Om)$, it should be regarded as an object of a dif\/ferential topological character attached to the symplectic manifold. In particular, the Schouten bracket is $\Symplecto(M, \Om)$-equivariant. The corresponding inf\/initesimal statement, the inf\/initesimal symplectomorphism equivariance of the Schouten bracket, is equivalent to the statement that the Lie derivative $\lie_{X}$ along $X \in \symplecto(M, \Om)$ is a derivation of $\shl \dum, \dum \shr$. Alternatively this is a consequence of the identity $\shl X^{\sflat}, \shl\al, \be\shr\shr = \lie_{X}\shl\al, \be\shr$ in conjunction with the Jacobi identity.

\subsection{}
Two Lie brackets $[\dum, \dum]$ and $\shl\dum, \dum\shr$ on a vector space $\g$ are \textit{compatible} if any linear combination of them is again a Lie bracket. For background on compatible Lie brackets see, for example, \cite{Bolsinov,Bolsinov-Borisov,Golubchik-Sokolov}, and the references therein. A~straightforward computation shows that compatibility of two Lie brackets is equivalent to the condition that each of the brackets is a~cocyle with respect to the other; this means $\shl\dum, \dum\shr$ is a cocyle of the cochain complex of $(\g, [\dum, \dum])$ with coef\/f\/icients in its adjoint representation, and likewise with $[\dum, \dum]$ and $\shl\dum, \dum\shr$ interchanged. From this second characterization it follows that if one of two Lie brackets is a~coboundary in the Lie algebra cohomology of the other, then the brackets are compatible.

Lemma~\ref{compatibilitylemma} gives an interpretation of $\sd^{\ast}$ in terms of the Lie algebra cohomology of $(\dum, \dum)$.

\begin{Lemma}\label{compatibilitylemma}
Let $(M, \Om)$ be a symplectic manifold. For $\nabla \in \symcon(M, \Om)$ the associated operator $\sd^{\ast} = \sd_{\nabla}^{\ast}$ satisfies:
\begin{enumerate}[$1.$]\itemsep=0pt
\item \label{sdastder} $\sd^{\ast}$ is a derivation of the symmetric product $\sprod$. That is,
\begin{gather*}
\sd^{\ast}(\al \sprod \be) = \sd^{\ast}\al \sprod \be + \al \sprod \sd^{\ast}\be \qquad \text{for} \quad \al \in\ssec^{k}(M)\quad \text{and}\quad \be \in \ssec^{l}(M).
\end{gather*}
\item \label{twobrackets}
The Schouten bracket $\shl\dum, \dum\shr$ is the coboundary of $\sd^{\ast}$ in the Lie algebra cochain complex of the algebraic bracket $(\dum, \dum)$ with coefficients in its adjoint representation. That is,
\begin{gather}\label{twobracketsidentity}
\shl\al, \be\shr = (\sd^{\ast}\al, \be) + (\al, \sd^{\ast}\be) - \sd^{\ast}(\al, \be)\qquad \text{for} \quad \al \in\ssec^{k}(M)\quad \text{and}\quad \be \in \ssec^{l}(M).
\end{gather}
Consequently, the Schouten bracket $\shl\dum,\dum\shr$ and the algebraic Lie bracket $(\dum, \dum)$ are compatible Lie brackets on $\hssec(M)$. Explicitly this means
\begin{align}
0 & =\cyclearg{\al, \be, \ga}\left[ (\shl\al, \be\shr, \ga) + \shl(\al, \be), \ga)\shr\right]\nonumber \\
& = (\shl\al, \be\shr, \ga) + (\shl\be, \ga\shr, \al) + (\shl\ga, \al\shr, \be)
  + \shl(\al, \be), \ga\shr+ \shl(\be, \ga), \al)\shr+ \shl(\ga, \al), \be\shr.\label{liecompatible}
\end{align}
\item For $\bnabla = \nabla + \Pi_{ij}\,^{k} \in \symcon(M, \Om)$ the associated operator $\sd_{\bnabla}^{\ast}$ is cohomologous to $\sd_{\nabla}^{\ast}$ in the Lie algebra cochain complex of the algebraic bracket $(\dum, \dum)$.
\end{enumerate}
\end{Lemma}

\begin{proof} The identity~\eqref{sdastder} follows from a straightforward computation. The identity~\eqref{twobrackets} follows by comparing~\eqref{schouten} with the identities
\begin{gather*}
\sd^{\ast}(\al, \be)_{i_{1}\dots i_{k+l-1}} = -kl\big(\al_{p(i_{1}\dots i_{k-1}}\nabla_{i_{k}}\be_{i_{k+1}\dots i_{k+l-1})}\,^{p} -\be_{p(i_{1}\dots i_{l-1}}\nabla_{i_{l}}\al_{i_{l+1}\dots i_{k+l-1})}\,^{p}\big),\\
(\al, \sd^{\ast}\be)_{i_{1}\dots i_{k+l-1}} = -k\al_{p(i_{1}\dots i_{k-1}}\nabla^{p}\be_{i_{k}\dots i_{k+l-1})} - kl \al_{p(i_{1}\dots i_{k-1}}\nabla_{i_{k}}\be_{i_{k+1}\dots i_{k+l-1})}\,^{p},
\end{gather*}
which follow straightforwardly from the def\/inition~\eqref{algebraicbracket}. Since $\shl\dum,\dum\shr$ is a coboundary with respect to $(\dum,\dum)$ it is automatically compatible with $(\dum, \dum)$. By claim~\eqref{twobrackets}, $\sd^{\ast}_{\bnabla} - \sd^{\ast}_{\nabla}$ is closed in the Lie algebra cochain complex of the algebraic bracket $(\dum, \dum)$. For $\Pi \in T_{\nabla}\symcon(M, \Om)$, $\bnabla = \nabla + \Pi, \nabla\in \symcon(M, \Om)$, and $\ga \in \ssec^{k}(M)$, there holds
\begin{gather}\label{sdasttransform}
\sd^{\ast}_{\bnabla}\ga_{i_{1}\dots i_{k+1}} = \sd^{\ast}_{\nabla}\ga_{i_{1}\dots i_{k+1}} + \tfrac{1}{3}(\ga, \Pi)_{i_{1}\dots i_{k+1}}.
\end{gather}
The identity~\eqref{sdasttransform} means that $\sd^{\ast}_{\bnabla} - \sd^{\ast}_{\nabla}$ is exact, that is that $\sd^{\ast}_{\bnabla}$ and $\sd^{\ast}_{\nabla}$ are cohomologous in the Lie algebra cochain complex of the algebraic bracket $(\dum, \dum)$.
\end{proof}

\begin{Lemma}\label{lieisolemma}
For $s, t \in \reat$ let $[\dum, \dum]_{s, t} = s(\dum, \dum) + t\shl\dum, \dum\shr$. If $s$, $\bar{s}$, $t$, and $\bar{t}$ are nonzero, the Lie brackets $[\dum, \dum]_{s, t}$ and $[\dum, \dum]_{\bar{s}, \bar{t}}$ on $\hssec(M)$ are isomorphic.
\end{Lemma}
\begin{proof}\looseness=-1
For $m \in \integer$ and $t \in \reat$, def\/ine $\Theta^{m}_{t}\colon \hssec(M) \to \hssec(M)$ by $\Theta^{m}_{t}(\al)_{k} = t^{k-m}\al_{k}$. The maps $\Theta_{t}^{m} \in \eno(\hssec(M))$ are linear isomorphisms such that $\Theta_{t}^{2}$ preserves~$(\dum, \dum)$ and $\Theta_{t}^{1}$ preserves~$\shl\dum, \dum\shr$. For $t, s \in \reat$, the linear isomorphism $\Psi_{s, t}\! = \!\Theta^{2}_{t^{-1}}\circ \Theta^{1}_{s}$ satisf\/ies $[\Psi_{s, t}(\al), \Psi_{s, t}(\be)]_{1, 1}\!=\!\Psi_{s, t}( [\al, \be]_{s, t})$, hence $\Psi_{s, t}$ is an isomorphism from $(\hssec(M), [\dum, \dum]_{s, t})$ to $(\hssec(M), [\dum, \dum]_{1, 1})$ if \mbox{$s, t \in \reat$}.
\end{proof}

The brackets $[\dum, \dum]_{s, t}$ interpolate between $(\dum, \dum)$ and $\shl\dum,\dum\shr$. Lemma~\ref{lieisolemma} means that linear combinations of $(\dum, \dum)$ and $\shl\dum, \dum\shr$ in which both brackets f\/igure nontrivially are all isomorphic. A convenient such linear combination is the bracket
\begin{gather}\label{compatbracket}
[\al, \be] = (\al, \be) + \shl\al, \be\shr.
\end{gather}
With respect to $[\dum, \dum]$, the subspace $\hssec_{2}(M)$ is a subalgebra of $\hssec(M)$, and the subspaces $\hssec_{k}(M) \subset \hssec_{2}(M)$ are Lie ideals. Since $\Psi_{s, t}$ preserves the f\/iltration of $\hssec(M)$ by the subspaces $\hssec_{k}(M)$, these, and other similar, statements that depend only on the f\/iltration, are valid also for the Lie brackets $[\dum, \dum]_{s, t}$.

Let $\al_{i} \in \ssec^{i}(M)$ be the projection onto $\ssec^{i}(M)$ of $\al \in \hssec(M)$.

\begin{Lemma}\label{hamseclemma}
Let $(M, \Om)$ be a symplectic manifold. With respect to the bracket $[\dum, \dum]$ of~\eqref{compatbracket} and the symmetric product $\sprod$, the subspace
\begin{gather}\label{hamsecdefined}
\hhamsec(M) = \big\{\al \in \hssec(M)\colon \al_{1} = -\sd^{\ast}\al_{0}\big\}
\end{gather}
is a Poisson subalgebra of $(\hssec(M), [\dum, \dum])$. For $k \geq 0$, let $\hhamsec_{k}(M) = \{\al \in \hhamsec(M)\colon \al_{i} = 0 \,\,\text{if}$ $i < k\}$ $($so $\hhamsec_{k}(M) = \hssec_{k}(M) $ for $k \geq 2)$. For $k \geq 0$ and $l \geq 2$, $[\hhamsec_{k}(M), \hhamsec_{l}(M)] \subset \hhamsec_{k+l-2}(M)$. In particular, for $k\geq 2$, $(\hhamsec_{k}(M), [\dum, \dum], \sprod)$ is a Poisson ideal in $(\hhamsec(M), [\dum, \dum], \sprod)$. Define $\pi\colon \hhamsec(M) \to \cinf(M)$ by $\pi(\al) = -\al_{0}$. The sequence
\begin{gather}\label{hamsplit}
\{0\} \longrightarrow \big(\hhamsec_{2}(M), [\dum, \dum], \sprod\big) \longrightarrow \big(\hhamsec(M), [\dum, \dum], \sprod\big) \stackrel{\pi}{\longrightarrow} \big(\cinf(M), \{\dum, \dum\}\big) \longrightarrow \{0\}
\end{gather}
of Poisson algebras is exact and $\iota\colon (\cinf(M), \{\dum, \dum\}) \to (\hhamsec(M), [\dum, \dum])$ defined by $\iota(f) = -f + \sd^{\ast}f$ is an injective Lie algebra homomorphism such that $\pi \circ \iota$ is the identity. That is, the sequence~\eqref{hamsplit} splits and $\hhamsec(M)$ is the semidirect product of the Lie subalgebra
\begin{gather*}
\hinf(M) = \iota\big(\cinf(M)\big) = \big\{\al \in \hssec(M)\colon \al_{1} = -\sd^{\ast}\al_{0},\, \al_{i} = 0 \,\, \text{for}\,\, i \geq 2\big\}
\end{gather*}
and the Poisson ideal $\hhamsec_{2}(M)$.
\end{Lemma}

\begin{proof}
Combining the $k = 0$ case of~\eqref{schouten} and the $k = 1$ case of~\eqref{algebraicbracket}, shows that
\begin{gather}\label{shalgf}
\shl f, \al\shr = -k\imt(\hm_{f})\al = (\sd^{\ast}f, \al),\qquad \text{for}\quad f \in \cinf(M) \quad\text{and}\quad \al \in \ssec^{k}(M).
\end{gather}
Alternatively,~\eqref{shalgf} is a special case of~\eqref{twobracketsidentity}. Let $\al, \be \in \hhamsec(M)$. By~\eqref{shalgf},
\begin{gather}
 (\al_{1}, \be_{j+1}) + (\al_{j+1}, \be_{1}) + \shl\al_{0}, \be_{j+1}\shr + \shl \al_{j+1}, \be_{0}\shr\nonumber\\
\qquad{}= - (\sd^{\ast}\al_{0}, \be_{j+1}) - (\al_{j+1}, \sd^{\ast}\be_{0}) + \shl\al_{0}, \be_{j+1}\shr + \shl \al_{j+1}, \be_{0}\shr = 0,\label{shalcancel}
\end{gather}
for all $j \geq 0$. By~\eqref{shalcancel},
\begin{align}
[\al, \be]_{0} & = (\al_{1}, \be_{1}) + \shl\al_{0}, \be_{1}\shr + \shl\al_{1}, \be_{0}\shr\nonumber \\
&= (\sd^{\ast}\al_{0}, \sd^{\ast}\be_{0}) - \shl\al_{0}, \sd^{\ast}\be_{0}\shr - \shl\sd^{\ast}\al_{0}, \be_{0}\shr
 = - (\sd^{\ast}\al_{0}, \sd^{\ast}\be_{0}) = -\{\al_{0}, \be_{0}\} ,\label{hamsec1}
\end{align}
and, by~\eqref{shalcancel} and~\eqref{sdschouten},
\begin{align}
[\al, \be]_{1} & = (\al_{1}, \be_{2}) + (\al_{2}, \be_{1}) + \shl\al_{0}, \be_{2}\shr + \shl\al_{1}, \be_{1}\shr + \shl\al_{2}, \be_{0}\shr \nonumber\\
& = \shl\sd^{\ast}\al_{0}, \sd^{\ast}\be_{0}\shr = \sd^{\ast}\{\al_{0}, \be_{0}\} = \sd^{\ast}(\sd^{\ast}\al_{0}, \sd^{\ast}\be_{0}) = -\sd^{\ast}[\al, \be]_{0} .\label{hamsec2}
\end{align}
By~\eqref{hamsec1} and~\eqref{hamsec2}, $(\hhamsec(M), [\dum, \dum])$ is a Lie subalgebra of $(\hssec(M), [\dum, \dum])$ and $\hhamsec_{2}(M)$ is a Lie ideal of $(\hhamsec(M), [\dum, \dum])$. That $(\hhamsec(M), [\dum, \dum], \sprod)$ is a Poisson subalgebra of $(\hssec(M), [\dum, \dum], \sprod)$ follows from
\begin{gather*}
\sd^{\ast}(\al_{0}\be_{0}) = -\al_{0}\be_{1} - \be_{0}\al_{1} = -(\al \sprod \be)_{1}.
\end{gather*}
Let $\al, \be \in \hhamsec(M)$ and let $j \geq 2$. Then,
\begin{gather}\label{hamsec3}
[\al, \be]_{j} = \sum_{i= 1}^{j+1}(\al_{i}, \be_{j+2-i}) + \sum_{i = 0}^{j+1}\shl \al_{i}, \be_{j+1-i}\shr
= \sum_{i= 2}^{j}(\al_{i}, \be_{j+2-i}) + \sum_{i = 1}^{j}\shl \al_{i}, \be_{j+1-i}\shr,
\end{gather}
where the second equality follows from~\eqref{shalcancel}. Suppose $k \geq 0$, $l \geq 2$, $\al \in \hhamsec_{k}(M)$, and $\be \in \hhamsec_{l}(M)$. If $j < k+l-2$ and $i \geq k$ then $j+ 1 - i < j + 2 - i \leq j + 2 -k < l$; hence every term in the sums in the last line of~\eqref{hamsec3} vanishes. This shows that $[\hhamsec_{k}(M), \hhamsec_{l}(M)] \subset \hhamsec_{k+l-2}(M)$. It also follows from~\eqref{hamsec3} that $\hinf(M)$ is a Lie subalgebra of $(\hhamsec(M), [\dum, \dum])$. By~\eqref{hamsec1},
\begin{gather*}
\pi([\al, \be]) =- [\al, \be]_{0} = (\sd^{\ast}\al_{0}, \sd^{\ast}\be_{0}) = \{\al_{0}, \be_{0}\} = \shl\pi(\al), \pi(\be)\shr,
\end{gather*}
so that $\pi$ is a Lie algebra homomorphism. For $f, g \in \cinf(M)$, by~\eqref{shalgf} and~\eqref{sdschouten},
\begin{align*}
[\iota(f), \iota(g)] & = [ -f + \sd^{\ast}f, -g + \sd^{\ast}g] = (\sd^{\ast}f, \sd^{\ast}g) - \shl f, \sd^{\ast}g\shr - \shl \sd^{\ast}f, g\shr + \shl \sd^{\ast}f, \sd^{\ast}g\shr\\
 &= -(\sd^{\ast}f, \sd^{\ast}g) + \sd^{\ast}\{f, g\} = -\{f, g\} + \sd^{\ast}\{f, g\} =\iota(\{f, g\}),
\end{align*}
showing that $\hinf(M) = \iota(\cinf(M))$ is a Lie subalgebra of $(\hhamsec(M), [\dum, \dum])$ that is isomorphic to $(\cinf(M), \{\dum, \dum\})$.
\end{proof}

\begin{Lemma}\label{derivationlemma}
Let $(M, \Om)$ be a symplectic manifold. For a symplectic vector field $X \!\in\! \symplecto(M, \Om)$,
\begin{enumerate}[$1)$]\itemsep=0pt
\item\label{xder1} the Lie derivative $\lie_{X}$ along $X$ is a derivation of the algebraic bracket $(\dum, \dum)$;
\item\label{xder2} the operator $I(X)\colon \hssec(M) \to \hssec(M)$ defined by $I(X)\al = k\imt(X)\al = -(X^{\sflat}, \al)$ for $\al \in \ssec^{k}(M)$ is a derivation of the Schouten bracket $\shl\dum, \dum\shr$.
\end{enumerate}
\end{Lemma}
\begin{proof}
Let $\al = X^{\sflat}$, $\be \in \ssec^{k}(M)$, and $\ga \in \ssec^{l}(M)$. By~\eqref{liecompatible} and the def\/inition of $I(X)$,
\begin{align}
\lie_{X}(\be, \ga)- (\lie_{X}\be, \ga) - (\be, \lie_{X}\ga)& = \shl \al, (\be, \ga)\shr - (\shl \al, \be\shr, \ga) - (\be, \shl \al, \ga\shr)\nonumber\\
& = -(\al, \shl \be, \ga\shr) + \shl (\al, \be), \ga\shr + \shl \be, (\al, \ga)\shr\nonumber \\
& = I(X)\shl\be, \ga\shr - \shl I(X)\be, \ga\shr - \shl\be, I(X)\ga \shr.\label{xdereq}
\end{align}
By~\eqref{xdereq},~\eqref{xder1} and~\eqref{xder2} are equivalent, so it suf\/f\/ices to prove either. As the claims are local in nature, it can be assumed $\al$ is exact. Taking $\al = -\sd^{\ast}f$, for $f \in \cinf(M)$, in~\eqref{xdereq} yields, by~\eqref{shalgf},
\begin{gather*}
(\sd^{\ast}f, \shl \be, \ga\shr) - \shl (\sd^{\ast}f, \be), \ga\shr - \shl \be, (\sd^{\ast}f,\ga)\shr = \shl f, \shl \be, \ga \shr\shr - \shl \shl f, \be\shr, \ga \shr - \shl \be, \shl f, \ga\shr\shr = 0. \tag*{\qed}
\end{gather*}\renewcommand{\qed}{}
\end{proof}

Let $\pi_{k}\colon \hssec(M) \to \hssec(M)$ be the projection with image $\hssec_{k}(M)$ def\/ined by $\pi_{k}\al = \al - \sum_{i = 0}^{k-1}\al_{i}$. Def\/ine \textit{truncated} brackets on $\hssec(M)$ by $(\al, \be)_{(k)} = (\pi_{k}\al, \pi_{k}\be)$ and $\shl\al, \be\shr_{(k)} = \shl\pi_{k}\al, \pi_{k}\be\shr$. If $k \geq 2$ then $2k - 2 \geq k$ so $(\al, \be)_{(k)} \in \hssec_{k}(M)$ and so $\pi_{k}(\al, \be)_{(k)} = (\al, \be)_{(k)}$; this identity suf\/f\/ices to show that $(\dum, \dum)_{(k)}$ satisf\/ies the Jacobi identity, so is a Lie bracket on $\hssec(M)$. The same argument shows that if $k \geq 1$ then $\shl\dum, \dum\shr_{(k)}$ is a Lie bracket on $\hssec(M)$. If $k \geq 1$ then $\hssec_{k+1}(M)$ is an ideal with respect to each of the brackets $\shl\dum, \dum\shr_{(k)}$ and $(\dum, \dum)_{(k+1)}$.
\begin{Lemma}
Let $(M, \Om)$ be a symplectic manifold. For $k \geq 1$, the Lie brackets $\shl\dum, \dum\shr_{(k)}$ and $(\dum, \dum)_{(k+1)}$ on $\hssec_{k+1}(M)$ are compatible.
\end{Lemma}

\begin{proof} For $\al, \be \in \hssec_{k+1}(M)$, by the def\/initions of $\shl \dum, \dum \shr_{(k)}$ and $(\dum, \dum)_{(k+1)}$, and~\eqref{twobracketsidentity},
\begin{align*}
\shl\al, \be\shr_{(k)} & = \shl\pi_{k}\al, \pi_{k}\be \shr= (\sd^{\ast}\pi_{k}\al, \pi_{k}\be) + (\pi_{k}\al, \sd^{\ast}\pi_{k}\be) - \sd^{\ast}(\pi_{k}\al, \pi_{k}\be)\\
& = (\pi_{k+1}\sd^{\ast}\al, \pi_{k+1}\be) + (\pi_{k+1}\al, \pi_{k+1}\sd^{\ast}\be) - \sd^{\ast}(\pi_{k+1}\al, \pi_{k+1}\be)\\
& = (\sd^{\ast}\al, \be)_{(k+1)} + (\al, \sd^{\ast}\be)_{(k+1)} - \sd^{\ast}(\al, \be)_{(k+1)},
\end{align*}
where the penultimate equality follows from $\sd^{\ast}\pi_{k} = \pi_{k+1}\sd^{\ast}$ and $\pi_{k}\pi_{k+1} = \pi_{k+1}$ (since $\al \in \hssec_{k+1}(M)$, the latter implies $\pi_{k}\al = \pi_{k+1}\al$). This shows that $\shl\dum, \dum\shr_{(k)}$ is a cocycle with respect to $(\dum, \dum)_{(k+1)}$.
\end{proof}

That an arbitrary vector f\/ield need not preserve $\Om$ suggests working with the subspace $\hzsec(M)$ of $\hssec(M)$, the degree one part of which comprises closed one-forms on $M$ (the symplectic dual $\symplecto(M, \Om)^{\sflat}$ of $\symplecto(M, \Om)$). (A related problem is that the $(\dum, \dum)$-bracket of a closed one-form with an element of $\ssec^{2}(M)$ need not be a closed one-form.) Def\/ine
\begin{gather}\label{hbsecdefined}
\hbsec(M) = \big\{\al \in \hssec(M)\colon \al_{1} \,\,\text{is exact}\big\} \subset \hzsec(M) = \big\{\al \in \hssec(M)\colon d\al_{1} = 0\big\}.
\end{gather}
If $\al, \be \in \hzsec(M)$ then $\shl \al, \be \shr_{(1)\,1} = \shl \pi_{1}\al, \pi_{1}\be\shr_{1} = \shl \al_{1}, \be_{1}\shr$ is closed, since the Schouten bracket of closed one-forms is closed. Hence $\hzsec(M)$ is a subalgebra of $\hssec(M)$ with respect to the bracket $\shl \dum, \dum\shr_{(1)}$. If $\al, \be \in \hzsec(M)$ then $(\al, \be)_{(2)} \in \hssec_{2}(M) = \hzsec(M)\cap \hssec_{2}(M)$, so $\hzsec(M)$ is a subalgebra of~$\hssec(M)$ with respect to the bracket $(\dum, \dum)_{(2)}$.

\begin{Lemma}\label{truncatedcompatiblelemma} Let $(M, \Om)$ be a symplectic manifold. The Lie brackets $\shl\dum, \dum\shr_{(1)}$ and $(\dum, \dum)_{(2)}$ on $\hzsec(M)$ are compatible and $\hbsec(M)$ is an ideal in $\hzsec(M)$ with respect to each of $\shl\dum, \dum\shr_{(1)}$ and $(\dum, \dum)_{(2)}$.
\end{Lemma}

\begin{proof}\allowdisplaybreaks
Let $\al, \be, \ga \in \hzsec(M)$. Then, using $\pi_{1}\al = \al_{1} + \pi_{2}\al$,
\begin{gather*}
\cyclearg{\al, \be, \ga}\big( (\shl \al, \be \shr_{(1)}, \ga)_{(2)} + \shl (\al, \be)_{(2)}, \ga\shr_{(1)}\big) \\
\qquad{} = \cyclearg{\al, \be, \ga}\big( (\pi_{2}\shl \pi_{1}\al, \pi_{1}\be\shr, \pi_{2}\ga) + \shl \pi_{1}(\pi_{2}\al, \pi_{2}\be) , \pi_{1}\ga\shr\big)\\
\qquad{} = \cyclearg{\al, \be, \ga}\left( (\pi_{2}\shl \al_{1} + \pi_{2}\al, \be_{1} + \pi_{2}\be\shr, \pi_{2}\ga) + \shl \pi_{2}(\pi_{2}\al, \pi_{2}\be) , \ga_{1} + \pi_{2}\ga\shr\right)\\
\qquad{} = \cyclearg{\al, \be, \ga}\big( (\shl \pi_{2}\al, \pi_{2}\be\shr, \pi_{2}\ga) + \shl (\pi_{2}\al, \pi_{2}\be) , \pi_{2}\ga\shr
+ (\pi_{2}\shl \al_{1}, \pi_{2}\be\shr, \pi_{2}\ga)  \\
\qquad\quad{}
+   (\pi_{2}\shl \pi_{2}\al, \be_{1} \shr, \pi_{2}\ga) + (\pi_{2}\shl \al_{1}, \be_{1}\shr, \pi_{2}\ga) + \shl \pi_{2}(\pi_{2}\al, \pi_{2}\be) , \ga_{1} \shr\big)\\
\qquad{} = \cyclearg{\al, \be, \ga}\left( \shl \al_{1}, \pi_{2}\be\shr, \pi_{2}\ga) +(\shl \pi_{2}\al, \be_{1} \shr, \pi_{2}\ga) + \shl (\pi_{2}\al, \pi_{2}\be) , \ga_{1} \shr\right)\\
\qquad{} = \cyclearg{\al, \be, \ga}\big((\lie_{\al_{1}^{\ssharp}}\pi_{2}\be, \pi_{2}\ga)
+ (\lie_{\al_{1}^{\ssharp}} \pi_{2}\be, \pi_{2}\ga) - \lie_{\al_{1}^{\ssharp}} (\pi_{2}\be, \pi_{2}\ga) \big) = 0,
\end{gather*}
the last equality because, by Lemma~\ref{derivationlemma}, the Lie derivative along a~symplectic vector f\/ield is a~derivation of the algebraic bracket $(\dum, \dum)$. This shows $\shl\dum, \dum\shr_{(1)}$ and $(\dum, \dum)_{(2)}$ are compatible. If $\al \in \hbsec(M)$ and $\be \in \hzsec(M)$, then $(\al, \be)_{(2)} \in \hssec_{2}(M) \subset \hbsec(M)$, and there is $f \in \cinf(M)$ such that $\al_{1} = -df$, so that
\begin{gather*}
\shl\al, \be\shr_{(1)\,1} = \shl\pi_{1}\al, \pi_{1}\be\shr_{1} = \shl\al_{1}, \be_{1} \shr = \shl \be_{1}, df\shr = \lie_{\be_{1}^{\ssharp}}df = d\lie_{\be_{1}^{\ssharp}}f.
\end{gather*}
Hence $\shl\al, \be\shr \in \hbsec(M)$.
\end{proof}

By Lemma~\ref{truncatedcompatiblelemma} the bracket
\begin{gather}\label{truncatedbracket}
[\dum, \dum]_{\trun} = (\dum, \dum)_{(2)} + \shl\dum, \dum\shr_{(1)}
\end{gather}
(the subscript $\trun$ is meant to suggest \textit{truncated}) is a Lie bracket on $\hzsec(M)$ and $\hbsec(M)$ is an ideal of $(\hzsec(M), [\dum, \dum]_{\trun})$. Because $[\al, \be]_{\trun} \in \hssec_{1}(M)$ for all $\al, \be \in \hzsec(M)$, $\hzsec_{1}(M)= \hzsec(M) \cap \hssec_{1}(M)$ and $\hbsec_{1}(M)= \hbsec(M) \cap \hssec_{1}(M)$ are ideals in $(\hzsec(M), [\dum, \dum]_{\trun})$.

Lemma~\ref{hamseccentrallemma} shows that the Lie algebra $(\hhamsec(M){,} [\dum, \dum])$ is a central extension of $(\hbsec_{1}(M){,} [\dum, \dum]_{\trun})$.

\begin{Lemma}\label{hamseccentrallemma}
Let $(M, \Om)$ be a symplectic manifold. The sequence
\begin{gather}\label{hamsplit2}
\{0\} \longrightarrow \rea \stackrel{\iota}{\longrightarrow} \big(\hhamsec(M), [\dum, \dum]\big) \stackrel{\pi_{1}}{\longrightarrow} \big(\hbsec_{1}(M), [\dum, \dum]_{\trun}\big) \longrightarrow \{0\}
\end{gather}
of Lie algebras is exact, where $\iota(c)$ is the constant function on $M$ equal to $-c$.
\end{Lemma}
\begin{proof}
For $\al \in \hhamsec(M)$, $\pi_{1}\al = 0$ if and only if $\al_{i} = 0$ for $i \geq 2$ and $\al_{0}$ is constant. As~$\pi_{1}$ maps $\hhamsec(M)$ surjectively onto $\hbsec_{1}(M)$, the content of the claim is that $\pi_{1}$ is a Lie algebra homomorphism. Let $\al, \be \in \hhamsec(M)$. The identity~\eqref{hamsec3} can be restated as
\begin{gather*}
\pi_{2}[\al, \be] = (\pi_{2}\al, \pi_{2}\be) + \pi_{2}\shl\pi_{1}\al, \pi_{1}\be\shr = \pi_{2} [\al, \be]_{\trun}.
\end{gather*}
The identity~\eqref{hamsec2} shows $[\al, \be]_{1} = \sd^{\ast}\{\al_{0}, \be_{0}\}$. On the other hand,
\begin{align*}
[\al, \be]_{\trun\,1} & = (\pi_{2}\al, \pi_{2}\be)_{1} + \shl \pi_{1}\al, \pi_{1}\be\shr_{1} = \shl \al_{1}, \be_{1}\shr\\
 & = \shl \sd^{\ast}\al_{0}, \sd^{\ast}\be_{0}\shr= \sd^{\ast}\{\al_{0}, \be_{0}\} = [\al, \be]_{1} .
\end{align*}
Hence $\pi_{1}[\al, \be] = \pi_{1} [\al, \be]_{\trun} = [\al, \be]_{\trun} = [\pi_{1}\al, \pi_{1}\be]_{\trun}$.
\end{proof}

The subspace $\hssec_{2}(M)$ is a Lie ideal of each of $(\hhamsec(M), [\dum, \dum])$ and $(\hbsec(M), [\dum, \dum]_{\trun})$, and the restrictions to $\hssec_{2}(M)$ of the brackets $[\dum, \dum]$ and $[\dum, \dum]_{\trun}$ coincide.

\section[Functionals on the symplectic af\/f\/ine space]{Functionals on the symplectic af\/f\/ine space $\boldsymbol{(\symcon(M, \Om), \sOm)}$\\ of symplectic connections}\label{functionalsection}

This section reviews the def\/initions and characterizations of the preferred and critical symplectic connections.

\subsection{}
 The Lie derivative of the af\/f\/ine connection $\nabla$ along $X$,
\begin{gather}\label{lienabla}
(\lie_{X}\nabla)_{ij}\,^{k} = \nabla_{i}\nabla_{j}X^{k} + X^{p}R_{pij}\,^{k}
\end{gather}
 is def\/ined to be the derivative at $t = 0$ of the pullback $\phi_{t}^{\ast}\nabla$ along the f\/low $\phi_{t}$ of $X$. Given $\nabla \in \symcon(M, \Om)$, def\/ine $\lop = \lop_{\nabla}\colon \Ga(\ctm) \to T_{\nabla}\symcon(M, \Om)$ by $\lop(X^{\sflat})_{ijk} = (\lie_{X}\nabla)_{(ijk)}$ for $X \in \Ga(TM)$. Since $2(\lie_{X}\nabla)_{i[jk]} = \nabla_{i}dX^{\sflat}_{jk}$,
\begin{gather*}
(\lie_{X}\nabla)_{ijk} = (\lie_{X}\nabla)_{(ijk)} + \tfrac{2}{3}\nabla_{(i}dX^{\sflat}_{j)k} = \lop(X^{\sflat})_{ijk} + \tfrac{2}{3}\nabla_{(i}dX^{\sflat}_{j)k}.
\end{gather*}
If \looseness=-1 $X \in \symplecto(M, \Om)$, then $(\lie_{X}\nabla)_{ijk}$ is completely symmetric, so, in this case, $\lop(X^{\sflat})_{ijk} = (\lie_{X}\nabla)_{ijk}$.
Def\/ine a linear operator $\sRo\colon \ssec^{3}(M)\to \Ga(\ctm)$ by $\sRo(\be)_{i} = \be^{abc}R_{iabc}$. The formal adjoint $\sRo^{\ast}\colon \Ga(\ctm) \to \ssec^{3}(M)$ is $\sRo^{\ast}(\al)_{ijk} = -\al^{p}R_{p(ijk)}$. For $\al \in \Ga(\ctm)$, $\lop(\al)$ can be rewritten as
\begin{gather*}
\lop(\al)_{ijk} = \nabla_{(i}\nabla_{j}\al_{k)} + \al^{p}R_{p(ijk)} = \big(\sd^{\ast\, 2}\al\big)_{ijk} -\sRo^{\ast}(\al)_{ijk}.
\end{gather*}
The induced action of $\ham(M, \Om)$ on $\symcon(M, \Om)$ is given by $\hop\colon \cinf(M) \to T_{\nabla}\symcon(M, \Om)$ def\/ined by
\begin{gather*}
\hop(f) = \lie_{\hm_{f}}\nabla = \lop(-df) = \lop(\sd^{\ast}f) = \big(\sd^{\ast\,3}f\big) - \sRo^{\ast}(\sd^{\ast} f).
\end{gather*}
Taking formal adjoints shows $\lop^{\ast} = -\sd^{2} - \sRo$ and $\hop^{\ast} = (\lop \sd^{\ast})^{\ast} = \sd \lop^{\ast} = -\sd^{3} - \sd\sRo$. Tracing~\eqref{lienabla} and using the Ricci identity twice yields that, for $\nabla \in \symcon(M, \Om)$,
\begin{align}
\sd \lop(X^{\sflat})_{ij} &=\nabla_{p}(\lie_{X}\nabla)_{ij}\,^{p} = \nabla_{p}\nabla_{i}\nabla_{j}X^{p} + X^{q}\nabla_{p}R_{qij}\,^{p} + R_{qij}\,^{p}\nabla_{p}X^{q}\nonumber\\
& = \nabla_{i}\nabla_{p}\nabla_{j}X^{p} + R_{iq}\nabla_{j}X^{q} + 2X^{q}\nabla_{[q}R_{i]j}\nonumber\\
& = \nabla_{i}\nabla_{j}\nabla_{p}X^{p} + R_{jq}\nabla_{i}X^{q} + R_{iq}\nabla_{j}X^{q}+ X^{q}\nabla_{q}R_{ij} + 2X^{q}\nabla_{i}R_{[jq]}\nonumber\\
&= (\lie_{X}\ric)_{ij} .\label{tracedhf}
\end{align}
Taking $X = \hm_{f}$ in~\eqref{tracedhf} yields $\sd \hop(f) = \lie_{\hm_{f}}\ric$.

\subsection{}\label{momentmapsection}
The \textit{Cahen--Gutt moment map} $\K\colon \symcon(M, \Om)\to \cinf(M)$ is def\/ined by
\begin{gather}\label{kdefined}
\K(\nabla) = -\sd^{2}\ric - \tfrac{1}{2}R^{ij}R_{ij}+ \tfrac{1}{4}R^{ijkl}R_{ijkl} = \nabla^{i}\nabla^{j}R_{ij} - \tfrac{1}{2}R^{ij}R_{ij}+ \tfrac{1}{4}R^{ijkl}R_{ijkl} .
\end{gather}
The map $\K$ was def\/ined by M. Cahen and S. Gutt in \cite{Cahen-Gutt} (see also~\cite{Bieliavsky-Cahen-Gutt-Rawnsley-Schwachhofer} or~\cite{Gutt-remarks}), where it was shown that $\K$ is a moment map for the action of $\Ham(M, \Om)$ on $\symcon(M, \Om)$. The map $\K(\nabla)$ has the form required by Tamarkin's theorem; its part quadratic in the curvature is a constant multiple of the contraction of the f\/irst Pontryagin form $\pon_{1}(\nabla)$ of $\nabla$ with $\Om_{2}$: $\K(\nabla) = -\sd^{2}\ric - (\pi^{2}/2)\pon_{1}(\nabla)_{p}\,^{p}\,_{q}\,^{q}$. As a consequence, when $M$ is compact the integral of $\K(\nabla)$ depends only on the cohomology class $[\Om_{n-2}]$ of $\Om_{n-2}$ and the f\/irst Pontryagin class $[\pon_{1}]$ of $M$,
\begin{gather}\label{intk}
\int_{M}\K(\nabla)\,\Om_{n} = -4\pi^{2}\int_{M} \pon_{1}\wedge \Om_{n-2} = -4\pi^{2}\lb [\pon_{1}] \cup [\Om_{n-2}], [M]\ra.
\end{gather}
(See \cite{Fox-critical} for details.) Since integration against $f\Om_{n}$ def\/ines a linear functional on $\ham(M, \Om)$ for any $f \in \cinf(M)$, whatever is the precise meaning of the dual space $\ham(M, \Om)^{\ast}$, this space contains $\ham(M, \Om)$ as a subspace. Hence $\K$ can be regarded as a map $\K\colon \symcon(M, \Om)\to \ham(M, \Om)^{\ast}$.

The f\/irst variation $\vr_{\Pi}\F(\nabla)$ of a functional $\F$ on $\symcon(M, \Om)$ at $\nabla \in \symcon(M, \Om)$ in the direction of $\Pi \in T_{\nabla} \symcon(M, \Om)$ is def\/ined by $\vr_{\Pi} \F(\nabla) = \tfrac{d}{dt}_{|t = 0}\F(\nabla + t\Pi)$. The essential content of Theorem~\ref{momentmaptheorem} is that $\vr_{\Pi} \K_{\nabla} = \H^{\ast}_{\nabla}(\Pi)$.

\begin{Theorem}[M.~Cahen and S.~Gutt, \cite{Cahen-Gutt}]\label{momentmaptheorem}
Let $(M, \Om)$ be a symplectic manifold. The map $\K\colon \symcon(M, \Om)\to \ham(M, \Om)^{\ast}$ is a moment map for the action of $\Ham(M, \Om)$ on $\symcon(M, \Om)$, equivariant with respect to the actions of $\Symplecto(M, \Om)$.
\end{Theorem}

\begin{proof}[Indication of proof] For a complete proof see \cite{Fox-critical}. Let $\bar{R}_{ijk}\,^{l}$ be the curvature of $\bnabla= \nabla + t\Pi_{ij}\,^{k}$ and label with a $\bar{\,}$ the tensors derived from it; for example $\bar{R}_{ij}$ is the Ricci curvature of $\bnabla$. Let $B(\Pi)_{ij} = \Pi_{ip}\,^{q}\Pi_{jq}\,^{p}$ and $(\Pi^{\ast}\Pi)_{i} = 3\sd B(\Pi)_{i} - \Pi^{abc}\nabla_{i}\Pi_{abc}$. Then
\begin{gather}
\label{riemvar} \bar{R}_{ijkl} = R_{ijkl} + 2t\nabla_{[i}\Pi_{j]kl} + 2t^{2}\Pi_{pl[i}\Pi_{j]k}\,^{p} = R_{ijkl} + t\dn\Pi_{ijkl} + t^{2}(\Pi \wedge \Pi)_{ijkl},\\
\label{ricvar} \bar{R}_{ij} = R_{ij} + t\nabla_{p}\Pi_{ij}\,^{p} -t^{2}\Pi_{ip}\,^{q}\Pi_{jq}\,^{p} = R_{ij} + t\sd\Pi_{ij} - t^{2}B(\Pi)_{ij},\\
\label{sdastricvary}
\sd_{\bnabla}^{\ast}\bric = \sd^{\ast}_{\nabla}\ric + t\big(\sd^{\ast}\sd\Pi + \tfrac{1}{3}(\ric, \Pi)\big)
+ t^{2}\big({-}\sd^{\ast}B(\Pi) + \tfrac{1}{3}(\sd\Pi, \Pi)\big) - \tfrac{1}{3}t^{3}(B(\Pi), \Pi),\\
\sd_{\nabla + t\Pi}\ric(\nabla + t\Pi)_{i} = \sd_{\nabla}\ric_{i} - t( \lop^{\ast}(\Pi)_{i} + \T(\Pi)_{i})
 - t^{2}\left(\sd B(\Pi)_{i} + \Pi_{i}\,^{pq}\sd\Pi_{pq} \right) + O\big(t^{3}\big),\!\!\!\!\label{rfvary}
\end{gather}
where $\T(\be)_{i} = \be^{abc}(R_{i(abc)} - \Om_{i(a}R_{bc)}) = \be_{i}\,^{pq}R_{pq} + \sRo(\be)_{i}$. Combining~\eqref{riemvar}--\eqref{rfvary} yields \begin{gather}\label{kvary}
\K(\nabla + t\Pi) = \K(\nabla) + t\hop^{\ast}(\Pi) + \tfrac{1}{2}t^{2}\sd(\Pi^{\ast}\Pi) + O\big(t^{3}\big),
\end{gather}
For $\phi \in \Symplecto(M, \Om)$, $\K(\phi^{\ast}\nabla) = \K(\nabla)\circ \phi$, so $\laa f \circ \phi, \K(\phi^{\ast}\nabla)\raa = \laa f, \K(\nabla)\raa$. This shows $\K$ is equivariant with respect to the action of $\Symplecto(M, \Om)$ on $\symcon(M, \Om)$ and its action on $\ham(M, \Om)^{\ast}$ induced by its action on $\cinf(M)$. For $\nabla \in \symcon(M, \Om)$, $\Pi \in T_{\nabla}\symcon(M, \Om)$, $X \in \symplecto(M, \Om)$, and $f \in \cinf_{c}(M)$ it follows from~\eqref{kvary} that $\vr_{\Pi} \K_{\nabla} = \H_{\nabla}^{\ast}(\Pi)$. Hence $\sOmega_{\nabla}(\hop(f), \Pi) = \laa f, \hop^{\ast}(\Pi)\raa = \vr_{\Pi} \laa f, \K(\nabla)\raa$, showing that $\K$ is a moment map.
\end{proof}

\begin{Lemma}\label{precriticallemma}
For any $f \in C^{4}(\rea)$, the Hamiltonian vector field on $(\symcon(M, \Om), \sOm)$ generated by $\ex_{f}(\nabla) = \int_{M}f(\K(\nabla))\,\Om_{n}$ is $\shm_{\ex_{f}} = -\H(f^{\prime}(\K(\nabla)))$.
\end{Lemma}
\begin{proof}
Calculating the f\/irst variation $\vr_{\Pi}\ex_{f}(\nabla)$ along $\Pi \in T_{\nabla}\symcon(M, \Om)$ using~\eqref{kvary} yields
\begin{gather}\label{varex1}
\vr_{\Pi}\ex_{f}  = \laa f^{\prime}(\K(\nabla)), \H^{\ast}(\Pi)\raa= \sOm_{\nabla}(\hop(f^{\prime}(\K(\nabla))), \Pi),
\end{gather}
from which the claim follows.
\end{proof}

A symplectic connection $\nabla \in \symcon(M, \Om)$ is \textit{critical} if it is a critical point, for arbitrary compactly supported variations $\al \in \ssec^{3}(M)$, of $\ex\colon \symcon(M, \Om)\to \cinf(M)$ def\/ined by $\ex(\nabla) = \int_{M} \K(\nabla)^{2}\,\Om_{n}$. Motivation for studying critical connections is given in \cite{Fox-critical}, where the author proposed this notion.

\begin{Corollary}[\cite{Fox-critical}]\label{criticallemma}
A symplectic connection $\nabla{\in}\symcon(M, \Om)$ is critical if and only if $\hop(\K(\nabla)){=}0$.
\end{Corollary}

\subsection{}\label{preferredsection}
Given $\nabla \in \symcon(M, \Om)$, there are written $\rice_{i}\,^{j}$ for the \textit{Ricci endomorphism} $R_{i}\,^{j}$ and $\rice^{\circ k}$ for its $k$th power as a f\/iberwise endomorphism. Since $(\rice^{\circ k})_{ij} = (-1)^{k+1}(\rice^{\circ k})_{ji}$, $(\rice^{\circ k})_{ij}$ is symmetric if $k$ is odd, and skew-symmetric if $k$ is even. In particular, $\tr \rice^{\circ 2k+1} = 0 $ and
\begin{gather*}
\tr \rice^{\circ 2k} = -R^{pq}\big(\rice^{\circ 2k-1}\big)_{pq} = (-1)^{k} R_{i_{1}i_{2}}R^{i_{2}i_{3}}R_{i_{3}i_{4}}\cdots R^{i_{2k-2}i_{2k-1}}R_{i_{2k-1}i_{2k}}R^{i_{2k}i_{1}}.
\end{gather*}
\begin{Remark}
The functionals $\tr \rice^{\circ s}$ play a key role in the averaging procedure used in \cite[Section~2]{Fedosov-atiyahbottpatodi}, where they are called \textit{cycles of length~$s$}.
\end{Remark}
For $1 \leq k$ and $\ric^{\circ k}_{ij} = (\rice^{\circ k})_{ij}$, def\/ine $\rc_{(k)}\colon \symcon(M, \Om) \to \rea$ by
\begin{align}
\rc_{(k)}(\nabla)& = -\tfrac{1}{2k}\int_{M} \tr\big(\rice^{\circ 2k}\big) \, \Om_{n} \nonumber\\
&= \tfrac{1}{2k}\int_{M}R^{pq}\big(\rice^{\circ 2k-1}\big)_{pq}\,\Om_{n} = \tfrac{1}{2k}\sOm_{\nabla}\big(\ric^{\circ 2k-1}, \ric\big) .\label{rckdefined}
\end{align}

\begin{Lemma}\label{tauthamlemma}
The Hamiltonian vector field $\shm_{\rc_{(k)}}$ on $(\symcon(M, \Om), \sOm)$ generated by $\rc_{(k)}$ is
\begin{gather*}
\shm_{\rc_{(k)}}(\nabla)_{abc} = -\sd^{\ast}\ric^{\circ 2k-1}_{abc} = \nabla_{(a}\big(\rice^{\circ 2k-1}\big)_{bc)}.
\end{gather*}
\end{Lemma}
\begin{proof}
For $\Pi \in T_{\nabla}\symcon(M, \Om)$, by~\eqref{ricvar} and~\eqref{rckdefined},
\begin{gather}\label{varrck}
\sOm_{\nabla}(\Pi, \shm_{\rc_{(k)}})  = \vr_{\Pi} \rc_{(k)}(\nabla) = -\tfrac{1}{2k}\vr_{\Pi}\int_{M}\tr\big(\rice^{\circ 2k}\big)\Om_{n}\\
\hphantom{\sOm_{\nabla}(\Pi, \shm_{\rc_{(k)}}) }{} = \tfrac{1}{2k}\vr_{\Pi}\int_{M}R^{pq}\big(\rice^{\circ 2k-1}\big)_{pq}\Om_{n}
= \laa \sd \Pi, \ric^{\circ 2k-1}\raa = -\sOm_{\nabla}\big( \Pi, \sd^{\ast}\ric^{\circ 2k-1}\big).\!\!\!\!\tag*{\qed}
\end{gather}
\renewcommand{\qed}{}
\end{proof}

\begin{Remark}Since, by def\/inition, $\nabla \in \symcon(M, \Om)$ is \textit{preferred} if it is critical for $\rc_{(1)}$, that is a~zero of $\shm_{\rc_{(1)}}$, Lemma~\ref{tauthamlemma} suggests regarding the zeros of $\shm_{\rc_{(k)}}$, for $k > 1$, as generalizations of preferred symplectic connections.
\end{Remark}

\begin{Remark} For any $\Psi \in \Symplecto(M, \Om)$, it follows from $\K(\Psi^{\ast}\nabla) = \K(\nabla)\circ \Psi$ and the similar equivariance of the Ricci tensor, that $\ex_{\phi}(\Psi^{\ast}\nabla) = \ex_{\phi}(\nabla)$ and $\rc_{(k)}(\Psi^{\ast}\nabla) = \rc_{(k)}(\nabla)$. Consequently, $\ex_{\phi}$ and $\rc_{(k)}$ are constant on $\Symplecto(M, \Om)$ orbits in $\symcon(M, \Om)$.
\end{Remark}

The Poisson bracket $\spl \F, \Gf \rpl$ of functionals $\F$ and $\Gf$ on $(\symcon(M, \Om), \sOm)$ is def\/ined to be the symplectic pairing
\begin{gather}\label{splbracketdefined}
\spl \F, \Gf \rpl = \sOm(\shm_{\F}, \shm_{\Gf})
\end{gather}
of the Hamiltonian vector f\/ields $\shm_{\F}$ and $\shm_{\Gf}$ generated by $\F$ and $\Gf$.

\begin{Lemma}\label{commutingflowslemma}
If $(M, \Om)$ is compact, the functionals $\rc_{(k)}$ and $\ex_{\phi}$ on $(\symcon(M, \Om), \sOm)$ Poisson commute for any $k \geq 1$ and any $\phi \in C^{4}(\rea)$.
\end{Lemma}
\begin{proof}
By def\/inition of the Poisson bracket $\spl\dum, \dum \rpl$ on $(\symcon(M, \Om), \sOm)$,~\eqref{tracedhf}, and Lemma~\ref{precriticallemma},
\begin{align*}
\spl\ex_{\phi}, \rc_{(k)}\rpl & = \sOm(\shm_{\ex_{\phi}}, \shm_{\rc_{(k)}}) = \sOm\big(\hop(\phi^{\prime}(\K)), \sd^{\ast}\ric^{\circ 2k-1}\big)\\
& = - \laa\sd\hop(\phi^{\prime}(\K)), \ric^{\circ 2k-1}\raa = -\laa \lie_{\hm_{\phi^{\prime}(\K)}}\ric, \ric^{\circ 2k-1}\raa \\
& = \int_{M}(\lie_{\hm_{\phi^{\prime}(\K)}}\ric)_{i}\,^{j}\big(\rice^{\circ 2k-1}\big)_{j}\,^{i} \, \Om_{n} = \tfrac{1}{2k}\int_{M} \lie_{\hm_{\phi^{\prime}(\K)}}\big(\rice^{\circ 2k}\big) \, \Om_{n} = 0,
\end{align*}
the last equality because the integral of a divergence vanishes.
\end{proof}

\begin{Remark}
 It would be interesting to know what can be said about existence of solutions to the Hamiltonian f\/lows
\begin{gather}\label{flows}
\tfrac{d}{dt}\nabla(t) = c \shm_{\rc_{(k)}}(\nabla(t)).
\end{gather}
Here $\nabla(t)$ is a path in $\symcon(M, \Om)$ and $c \in \rea$. By Lemma~\ref{commutingflowslemma}, $\ex$ is constant along the f\/lows~\eqref{flows}.
\end{Remark}

\section[A symplectic af\/f\/ine action]{A symplectic af\/f\/ine action of $\boldsymbol{(\hhamsec(M), [\dum, \dum])}$ on $\boldsymbol{(\symcon(M, \Om), \sOm)}$}\label{atiyahbottsection}

\subsection{}\label{gaugeactionsection}
Recall from Remark~\ref{framebundleremark} that $\gge(\sfr)$ is the group of gauge transformations of the linear symplectic frame bundle $\sfr \to M$. The action of $\gge(\sfr)$, by pullback, on principal connections induce the right action, $\tcon(M, \Om) \times \gge(\sfr) \to \tcon(M, \Om)$, on the associated covariant derivatives given by
\begin{gather}\label{standardgaugeaction}
g\cdot \nabla = \nabla + (g^{-1})_{p}\,^{k}\nabla_{i}g_{j}\,^{p} = \nabla - g^{kp}\nabla_{i}g_{jp},
\end{gather}
where $g_{iq}g_{j}\,^{q} = g_{i}\,^{p}g_{j}\,^{q}\Om_{pq} = \Om_{ij}$ and $(g^{-1})_{ij} = -g_{ji}$. The covariant derivatives associated with principal connections on $\sfr$ preserve the symplectic structure, but need not be torsion-free, and the action of $\gge(\sfr)$ does not preserve the torsion, for, by~\eqref{standardgaugeaction}, if~$\nabla$ is torsion-free, then $\nabla^{g}$ is torsion-free if and only if $\nabla_{[i}g_{j]}\,^{k} = 0$. This need not be true in general. The action~\eqref{standardgaugeaction} is af\/f\/ine (as in Section~\ref{affineactionsection}). Its linear part $\lin(g)\colon T_{\nabla}\tcon(M, \Om) \to T_{g\cdot \nabla}\tcon(M, \Om)$ is $(\lin(g)\Pi)_{ijk} = (g^{-1})_{pk}\Pi_{iq}\,^{p}g_{j}\,^{q} = g_{j}\,^{p}g_{k}\,^{q}\Pi_{ipq}$, and there follows $\sOm(\lin(g)\al, \lin(g)\be) = \sOm(\al, \be)$, showing that the action of $\gge(\sfr)$ on $\tcon(M, \Om)$ is symplectic af\/f\/ine. However, there is no obvious way to modify the action of $\gge(\sfr)$ on $\tcon(M, \Om)$ to obtain an action on $\symcon(M, \Om)$.

\subsection{}\label{yangmillssection}
The curvature is a moment map for the gauge group action on the space of principal connections on a principal bundle over a symplectic manifold. This goes back to \cite{Atiyah-Bott}, where this construction is used in the special case of a surface. In the present context this can be realized concretely as follows. Given $\al \in \ssec^{2}(M)$, regarded as an element of the Lie algebra $\lge(\sfr)$ of inf\/initesimal gauge transformations (see Remark~\ref{framebundleremark}), regard the curvature of $\nabla\in \tcon(M, \Om)$ as a $\lge(\sfr)$-valued two-form $\curv$ and write $\tr (\curv\al) = R_{ijp}\,^{q}\al_{q}\,^{p}$ for the two-form resulting from tracing the composition of the curvature with the endomorphism $\al_{i}\,^{j}$. However the dual $\lge(\sfr)^{\ast}$ is understood, it includes $\Ad(\sfr)$-valued $2n$-forms.

\begin{Theorem}\label{abtheorem}
On a compact symplectic manifold $(M, \Om)$, the map $\fc\colon \tcon(M, \Om) \to \lge(\sfr)^{\ast}$ defined by $\fc(\nabla) = \curv \wedge \Om_{n-1}$ is a moment map for the action of the gauge group $\gge(\sfr)$ on $\tcon(M, \Om)$.
\end{Theorem}
\begin{proof}
By~\eqref{twokformwedge}, $\fc(\nabla) = \tfrac{1}{2}R_{p}\,^{p}\,_{ij}\Om_{n}$, so
\begin{gather}\label{fctcon0}
\laa \fc(\nabla), \al \raa = \int_{M} \tr(\curv\al)\wedge \Om_{n-1} = -\tfrac{1}{2}\int_{M} \al^{ij}R_{p}\,^{p}\,_{ij}\,\Om_{n},
\end{gather}
If $\nabla$ is a connection preserving a volume form $\mu$ and having torsion $\tau_{ij}\,^{k}$, then the divergence $\nabla_{p}X^{p}$ def\/ined by the connection and the divergence $(\lie_{X}\mu)/\mu$ def\/ined by the volume form are related by $(\lie_{X}\mu)/\mu = \nabla_{p}X^{p} + X^{p}\tau_{pq}\,^{q}$. For $\nabla \in \tcon(M, \Om)$ there results
\begin{gather}\label{tip}
\int_{M}\big( \nabla_{p}X^{p} + X^{p}\tau_{pq}\,^{q}\big)\,\Om_{n} = \int_{M}\lie_{X}\Om_{n} = 0.
\end{gather}
The f\/irst variation of the curvature of $\nabla \in \tcon(M, \Om)$ along $\Pi \in T_{\nabla}\tcon(M, \Om)$ is $(\vr_{\Pi}R)_{ijkl} = \dn\Pi_{ijkl} + \tau_{ij}\,^{p}\Pi_{pkl}$. Note that $0 = d\Om_{ijk} - 3\nabla_{[i}\Om_{jk]} = \tau_{[ijk]}$. Tracing this yields $2\tau_{ip}\,^{p} = -\tau_{p}\,^{p}\,_{i}$. Viewing $\al_{ij} \in \ssec^{2}(M)$ as an element of $\lge(\sfr)$ and dif\/ferentiating the action of $\exp(t\al)$ on $\tcon(M, \Om)$ yields the vector f\/ield $\X^{\al}(\nabla) = \tfrac{d}{dt}_{t = 0}(\exp(-t\al)\cdot \nabla) = -\nabla_{i}\al_{jk}$. Applying the preceding observations to~\eqref{fctcon0} and integrating by parts using~\eqref{tip} yields
\begin{align*}
\vr_{\Pi}\laa \fc(\nabla), \al \raa &= -\tfrac{1}{2}\vr_{\Pi}\int_{M} \al^{ij}R_{p}\,^{p}\,_{ij}\,\Om_{n} = \int_{M}\big(\al^{ij}\nabla^{p}\Pi_{pij} - \tfrac{1}{2}\tau_{q}\,^{qp}\Pi_{pij}\al^{ij}\big)\,\Om_{n}\\
& = \int_{M}\big( \Pi^{ijk}\nabla_{i}\al_{jk} - \nabla_{p}\big(\al^{ij}\Pi^{p}\,_{ij}\big) - \al^{ij}\Pi^{p}\,_{ij}\tau_{pq}\,^{q}\big)\,\Om_{n}\\
& = \int_{M}\Pi^{ijk}\nabla_{i}\al_{jk}\,\Om_{n} = - \sOm( \X^{\al}, \Pi).
\end{align*}
This shows that $\X^{\al}$ is the Hamiltonian vector f\/ield on $\tcon(M, \Om)$ generated by $\laa \fc(\nabla), \al\raa$. Since~$\fc$ is by construction $\gge(\sfr)$-equivariant, this shows $\fc$ is a moment map.
\end{proof}

\begin{Lemma}\label{alhamlemma}
For $\al \in \ssec^{2}(M)$, define $\scal_{\al}\colon \symcon(M, \Om) \to \rea$ by $\scal_{\al}(\nabla) = \laa \al, \ric(\nabla)\raa$. For a~compactly supported $\al \in\ssec^{2}(M)$ the Hamiltonian vector field generated on $(\symcon(M, \Om), \sOm)$ by $\scal_{\al}$ is $\shm_{\scal_{\al}}= -\sd^{\ast}\al$.
\end{Lemma}
\begin{proof}
By~\eqref{ricvar}, $\sOm_{\nabla}(\Pi, \shm_{\scal_{\al}}) = \vr_{\Pi} \scal_{\al}(\nabla) = \sOm_{\nabla}(\al, \sd_{\nabla}\Pi) = -\sOm_{\nabla}( \Pi, \sd_{\nabla}^{\ast}\al)$.
\end{proof}

If $\nabla \in \symcon(M, \Om)$, then $\fc(\nabla) = (1/2)R_{p}\,^{p}\,_{ij}\Om_{n} = R_{ij}\Om_{n}$, because $R_{p}\,^{p}\,_{ij} = 2R_{ij}$. (Note that the identity $R_{p}\,^{p}\,_{ij} = 2R_{ij}$ fails for $\nabla \in \tcon(M, \Om)$ having nonvanishing torsion.) Hence, by~\eqref{fctcon0}, $\laa \fc(\nabla), \al \raa = -\laa \al, \ric(\nabla)\raa$ for $\nabla \in \symcon(M, \Om)$. That is the restriction of $-\laa \fc(\nabla), \al \raa$ to $\symcon(M, \Om)$ equals the functional $\sR_{\al}$ def\/ined in Lemma~\ref{alhamlemma}. For $\Pi \in \Ga(\ctm \tensor S^{2}(\ctm))$ let $\Q(\Pi)_{ijk} = \Pi_{(ijk)}$ and $\Q^{\perp}(\Pi)_{ijk} = (\Id - \Q)(\Pi)_{ijk} = \tfrac{2}{3}(\Pi_{ijk} - \Pi_{(jk)i})$ be the projections onto $\ssec^{3}(M)$ and its complement. The Hamiltonian vector f\/ield $\shm_{\sR_{\al}}$ on $\symcon(M, \Om)$ is the projection $\Q(-\X^{\al}) = \sd^{\ast}\al$ onto $T\symcon(M, \Om)$ of the Hamiltonian vector f\/ield of $-\laa \fc(\nabla), \al \raa$ on $\tcon(M, \Om)$.

Recall from~\eqref{splbracketdefined} the def\/inition of the Poisson bracket $\spl\dum, \dum\rpl$ on $(\symcon(M, \Om), \sOm)$.
\begin{Lemma}\label{poissonlemma}
Let $(M, \Om)$ be a compact symplectic manifold. For $\al, \be \in \ssec^{2}(M)$, the Poisson bracket $\spl\scal_{\al}, \scal_{\be}\rpl$ of $\scal_{\al}$ and $\scal_{\be}$ is
\begin{gather*}
\spl\scal_{\al}, \scal_{\be}\rpl = -\tfrac{2}{3}\laa \sd \al, \sd\be\raa + \tfrac{1}{3}\laa \al \twprod \be, \curv \raa + \tfrac{1}{3}\scal_{(\al, \be)},
\end{gather*}
where $\al \twprod \be \in \Wl^{2, 2}(\ctm)$ is defined by $(\al \twprod \be)_{ijkl} = 2\al_{k[i}\be_{j]l} + 2\al_{l[i}\be_{j]k} = - (\be \twprod \al)_{ijkl}$.
The Hamiltonian vector field on $(\symcon(M, \Om), \sOm)$ generated by $\spl\scal_{\al}, \scal_{\be}\rpl$ is
\begin{gather}\label{hamalbe}
\shm_{\spl\scal_{\al}, \scal_{\be}\rpl} = -\tfrac{1}{3}(\sd^{\ast}(\al, \be) + \shl\al, \be\shr) = \tfrac{1}{3}\big(\shm_{\scal_{(\al, \be)}} - \shl\al, \be\shr\big).
\end{gather}
\end{Lemma}
\begin{proof}
Since $\curv_{ijkl} \in \Wl^{2, 2}(\ctm)$, the pairing~\eqref{wpairing} takes the form
\begin{gather*}
\laa \al \twprod \be, \curv \raa = 2\int_{M}\al^{il}\be^{jk}R_{ijkl}\,\Om_{n}.
\end{gather*}
By~\eqref{sdsdast} of Lemma~\ref{sdcommutationlemma},
\begin{gather}\label{ssbc}
\sd \sd^{\ast}\be_{ij} = -\tfrac{2}{3}\sd^{\ast}\sd\be_{ij} - \tfrac{2}{3}R^{p}\,_{(ij)}\,^{q}\be_{pq} + \tfrac{1}{3}(\be, \ric)_{ij}.
\end{gather}
By def\/inition of the Poisson bracket $\spl\scal_{\al}, \scal_{\be}\rpl$, Lemma~\ref{alhamlemma},~\eqref{ssbc}, and~\eqref{algliepairing},
\begin{align}
\spl\scal_{\al}, \scal_{\be}\rpl &= \sOm(\shm_{\scal_{\al}}, \shm_{\scal_{\be}}) = \laa \sd^{\ast}\al, \sd^{\ast}\be\raa = \laa \al, \sd \sd^{\ast}\be \raa \nonumber\\
& = -\tfrac{2}{3}\laa \al, \sd^{\ast}\sd\be\raa - \tfrac{2}{3}\int_{M}R^{p}\,_{(ij)}\,^{q}\be_{pq}\al^{ij}\,\Om_{n} + \tfrac{1}{3}\laa \al, (\be, \ric)\raa\nonumber\\
& = -\tfrac{2}{3}\laa \sd \al, \sd\be\raa + \tfrac{1}{3}\laa \al \twprod \be, \curv \raa + \tfrac{1}{3}\laa (\al, \be), \ric\raa.\label{pab1}
\end{align}
By the second equality of~\eqref{pab1} combined with~\eqref{sdasttransform},~\eqref{algliepairing}, and~\eqref{twobrackets} of Lemma~\ref{compatibilitylemma},
\begin{align*}
\sOm(\Pi, \shm_{\spl\scal_{\al}, \scal_{\be}\rpl})& = \vr_{\Pi}\spl\scal_{\al}, \scal_{\be}\rpl = \tfrac{1}{3}\laa(\al, \Pi), \sd^{\ast}\be\raa + \tfrac{1}{3}\laa \sd^{\ast}\al, (\be, \Pi)\raa\\
& = \tfrac{1}{3}\laa (\al, \sd^{\ast}\be) + (\sd^{\ast}\al, \be), \Pi\raa = -\tfrac{1}{3}\laa\Pi, \sd^{\ast}(\al, \be) + \shl\al, \be\shr\raa,
\end{align*}
which with Lemma~\ref{alhamlemma} shows~\eqref{hamalbe}.
\end{proof}

The f\/low of $-\scal_{\al}$ has the form $\nabla \to \nabla + s\sd^{\ast}\al$. Lemma~\ref{poissonlemma} can be viewed as showing that $\pi(\al, \nabla) = \al\cdot \nabla = \nabla + s\sd^{\ast}\al$ does not def\/ine an af\/f\/ine action of $(\ssec^{2}(M),(\dum, \dum)) \simeq \lge(\sfr)$ on $\symcon(M, \Om)$, and that the obstruction is given by the Schouten bracket. For $\al, \be \in\ssec^{2}(M)$, using~\eqref{twobracketsidentity} and~\eqref{sdasttransform} yields
\begin{align*}
\al\cdot (\be \cdot \nabla) - \be\cdot(\al \cdot \nabla) &= \tfrac{1}{3}s^{2}\left((\al, \sd^{\ast}\be) + (\sd^{\ast}\al, \be)\right) = \tfrac{1}{3}s^{2}\left(\sd^{\ast}(\al, \be) + \shl\al, \be\shr\right)\\
& = (\al, \be)\cdot \nabla - \nabla + \tfrac{1}{3}s(s-3)\sd^{\ast}(\al, \be) + \tfrac{1}{3}s^{2}\shl\al, \be\shr.
\end{align*}
Hence the closest $\al \cdot \nabla$ comes to def\/ining an af\/f\/ine action is when $s = 3$, in which case
\begin{gather}\label{albeschouten}
\big(\al\cdot (\be \cdot \nabla) - \be\cdot(\al \cdot \nabla)\big)- \big( (\al, \be)\cdot \nabla - \nabla\big) = 3\shl\al, \be\shr.
\end{gather}
The identity~\eqref{albeschouten} is essentially the same as~\eqref{hamalbe}. It yields a conceptual explanation of claim~\eqref{twobrackets} of Lemma~\ref{compatibilitylemma}, that the Schouten bracket is a coboundary with respect to the algebraic bracket $(\dum, \dum)$, as it exhibits $\shl\dum,\dum\shr$ as the obstruction to making $(\ssec^{2}(M),(\dum, \dum))$ act on an af\/f\/ine space. On the other hand, it means that $(\ssec^{2}(M),(\dum, \dum))$ can be extended by $\ssec^{3}(M)$, using $\shl\dum, \dum\shr$ to produce a Lie algebra acting on $\symcon(M, \Om)$.
The idea developed in the remainder of this section and in Section~\ref{symplecticaffinesection} is that, although $\lge(\sfr)$ does not act on $\symcon(M, \Om)$, and so $\fc$ (and correspondingly~$\ric$) cannot be viewed as a moment map on $\symcon(M, \Om)$, the action of $\lge(\sfr)$ can be extended in such a~way that a functional on $\symcon(M, \Om)$ constructed from $\ric$ yields a moment map for the action of the extended Lie algebra on $\symcon(M, \Om)$, and the actions of $\ham(M, \Om)$ and $\lge(\sfr)$ can be combined and further extended in such a way that a functional on $\symcon(M, \Om)$ constructed from $\K$ and $\ric$ yields a moment map for the action of the extended Lie algebra on $\symcon(M, \Om)$.

As a vector space $\ssec^{2}(M)\oplus\ssec^{3}(M)$ is identif\/ied with the quotient $\hssec_{2}(M)/\hssec_{4}(M)$. Since $\hssec_{4}(M)$ is an ideal of the Lie algebra $(\hssec_{2}(M), [\dum, \dum])$ for the Lie bracket $[\dum, \dum]$ def\/ined in~\eqref{compatbracket}, this bracket descends to a Lie bracket, also denoted $[\dum, \dum]$, on $\ssec^{2}(M)\oplus\ssec^{3}(M)$, given explicitly by
\begin{gather}\label{firstextended}
[\al, \be] = (\al_{2}, \be_{2}) + \shl\al_{2}, \be_{2}\shr + (\al_{2}, \be_{3}) - (\be_{2}, \al_{3})
\end{gather}
for $\al, \be \in \ssec^{2}(M)\oplus\ssec^{3}(M)$. For $\al \in\ssec^{2}(M)\oplus\ssec^{3}(M)$ and $\nabla \in \symcon(M, \Om)$, setting
\begin{gather*}
\al\cdot \nabla = \nabla + 3\sd^{\ast}\al_{2} + 3\al_{3},
\end{gather*}
yields a symplectic af\/f\/ine action on $\symcon(M, \Om)$ of the Lie algebra $(\ssec^{2}(M)\oplus\ssec^{3}(M), [\dum, \dum])$, where $[\dum,\dum]$ is the Lie bracket def\/ined by~\eqref{firstextended}. This is a special case of the more general Lemma~\ref{hbseclemma}.

\begin{Lemma}\label{hbseclemma}
Let $(M, \Om)$ be a symplectic manifold. Let $(\hbsec_{1}(M), [\dum,\dum]_{\trun})$ be the Lie algebra defined by~\eqref{hbsecdefined} and~\eqref{truncatedbracket}. The map $\pi^{\hbsec}\colon \hbsec_{1}(M) \times \symcon(M, \Om) \to \symcon(M, \Om)$ defined by
\begin{gather}\label{balnabla}
\pi^{\hbsec}(\al, \nabla) = \al \cdot \nabla = \nabla + \lop_{\nabla}(\al_{1}) + 3\sd^{\ast}_{\nabla}\al_{2} + 3\al_{3},
\end{gather}
for $(\al, \nabla) \in \hbsec_{1}(M) \times \symcon(M, \Om)$, is a symplectic affine action of $(\hbsec_{1}(M), [\dum, \dum]_{\trun})$ on $(\symcon(M, \Om), \sOm)$. The stabilizer of $\nabla \in \symcon(M, \Om)$ is the Lie subalgebra
\begin{gather}\label{bsecstab}
\hbsec_{1}(M)_{\nabla} = \big\{ \al \in \hbsec(M)\colon 3\al_{3} = - \lop_{\nabla}(\al_{1}) - 3\sd^{\ast}_{\nabla}\al_{2}\big\}.
\end{gather}
\end{Lemma}
\begin{proof}
Consider the more general map $\hbsec_{1}(M) \times \symcon(M, \Om) \to \symcon(M, \Om)$ def\/ined by
\begin{gather}\label{balnablagen}
\al \cdot \nabla = \nabla + q\lop_{\nabla}(\al_{1}) + r\sd^{\ast}_{\nabla}\al_{2} + s\al_{3},
\end{gather}
for $q, r, s \in \rea$. Let $\rho\colon (\hbsec_{1}(M), [\dum, \dum]_{\trun}) \to \aff(\symcon(M, \Om))$ be the associated map given by $\rho(\al)\nabla = \al \cdot \nabla - \nabla$. From the identities
\begin{gather*}
\sd^{\ast}_{\nabla + \Pi}\al_{2} = \sd^{\ast}_{\nabla}\al_{2} + \tfrac{1}{3}(\al_{2}, \Pi),\\
\lop_{\nabla + \Pi}(\al_{1}) = \lie_{\al_{1}^{\ssharp}}(\nabla + \Pi) = \lie_{\al_{1}^{\ssharp}}\nabla+ \lie_{\al_{1}^{\ssharp}}\Pi = \lop_{\nabla}(\al_{1}) + \shl\al_{1}, \Pi\shr,
\end{gather*}
(the f\/irst is~\eqref{sdasttransform}) it follows that the map $\al \to \lin(\rho(\al))$ associating with $\al$ the linear endomorphism
\begin{gather*}
\lin(\rho(\al))\Pi = q\shl \al_{1}, \Pi\shr + \tfrac{r}{3}(\al_{2}, \Pi),
\end{gather*}
of $\ssec^{3}(M)$, satisf\/ies $\rho(\al)(\nabla + \Pi) - \rho(\al)\nabla = \lin(\rho(\al))\Pi$. This shows that $\rho(\al)$ is an af\/f\/ine transformation of $\symcon(M, \Om)$ with linear part $\lin(\rho(\al))$, and so, by Lemma~\ref{affineactionlemma}, $\al \cdot$ is an af\/f\/ine action with linear part $\lin(\al \cdot) = \Id + \lin(\rho(\al))$. By the $\symplecto(M, \Om)$-equivariance of $\sOm$ and~\eqref{algliepairing}, for $\Pi, \Upsilon \in T_{\nabla}\symcon(M, \Om)$,
\begin{gather*}
\sOm_{\rho(\al)\nabla}(\lin(\rho(\al))\Pi, \Upsilon) + \sOm_{\rho(\al)\nabla}(\Pi, \lin(\rho(\al))\Upsilon)\\
\qquad{}= q\left(\laa \shl \al_{1}, \Pi \shr, \Upsilon\raa + \laa \Pi, \shl \al_{1}, \Upsilon \shr\raa \right) + \tfrac{r}{3}\left(\laa (\al_{2}, \Pi), \Upsilon\raa + \laa \Pi, ( \al_{2}, \Upsilon )\raa \right) = 0.
\end{gather*}
This shows that $\rho(\al)$ is symplectic af\/f\/ine for all $\al \in \hbsec_{1}(M)$. For $\al, \be \in \hbsec_{1}(M)$,
\begin{gather*}
\lin(\rho(\al))\rho(\be)\nabla  = q\shl \al_{1}, \rho(\be)\nabla\shr + \tfrac{r}{3}(\al_{2}, \rho(\be)\nabla)\\
\hphantom{\lin(\rho(\al))\rho(\be)\nabla}{} = q\shl \al_{1}, q\lop_{\nabla}(\be_{1}) + r\sd^{\ast}_{\nabla}\be_{2} + s\be_{3} \shr + \tfrac{r}{3}(\al_{2}, q\lop_{\nabla}(\be_{1}) + r\sd^{\ast}_{\nabla}\be_{2} + s\be_{3})\\
\hphantom{\lin(\rho(\al))\rho(\be)\nabla}{} = q^{2} \shl \al_{1}, \lop_{\nabla}(\be_{1}) \shr
+ qr\left(\shl \al_{1}, \sd^{\ast}_{\nabla}\be_{2} \shr + \tfrac{1}{3}(\al_{2}, \lop_{\nabla}(\be_{1}))\right)\\
\hphantom{\lin(\rho(\al))\rho(\be)\nabla=}{}
+ qs \shl \al_{1}, \be_{3} \shr + \tfrac{r^{2}}{3}(\al_{2}, \sd^{\ast}_{\nabla}\be_{2}) + \tfrac{rs}{3}(\al_{2}, \be_{3}).
\end{gather*}
Hence
\begin{gather}
\lin(\rho(\al))\rho(\be)\nabla - \lin(\rho(\be))\rho(\al)\nabla = q^{2}\big( \shl \al_{1}, \lop_{\nabla}(\be_{1}) \shr + \shl \lop_{\nabla}(\al_{1}), \be_{1} \shr\big)\nonumber\\
{}+ qr\big(\shl \al_{1}, \sd^{\ast}_{\nabla}\be_{2} \shr + \shl \sd^{\ast}_{\nabla}\al_{2}, \be_{1}\shr + \tfrac{1}{3}(\al_{2}, \lop_{\nabla}(\be_{1})) + \tfrac{1}{3}(\lop_{\nabla}(\al_{1}), \be_{2}))\big)\nonumber\\
{}+ qs \big(\shl \al_{1}, \be_{3} \shr + \shl \al_{3}, \be_{1}\shr\big) + \tfrac{r^{2}}{3}\big((\al_{2}, \sd^{\ast}_{\nabla}\be_{2}) + (\sd^{\ast}_{\nabla}\al_{2}, \be_{2})\big) + \tfrac{rs}{3}\big((\al_{2}, \be_{3}) + (\al_{3}, \be_{2})\big).\label{brho1}
\end{gather}
Since $d\al_{1} = 0 = d\be_{1}$, $\al_{1}^{\ssharp}, \be_{1}^{\ssharp} \in \symplecto(M, \Om)$, and
\begin{gather}\label{poisslop}
\shl \al_{1}, \lop_{\nabla}(\be_{1}) \shr + \shl \lop_{\nabla}(\al_{1}), \be_{1} \shr = [\lie_{\al_{1}^{\ssharp}},\lie_{\be_{1}^{\ssharp}}]\nabla = \lop_{\nabla}(\shl \al_{1}, \be_{1}\shr).
\end{gather}
Using~\eqref{poisslop} and~\eqref{twobracketsidentity} to simplify~\eqref{brho1} yields
\begin{gather}
\lin(\rho(\al))\rho(\be)\nabla - \lin(\rho(\be))\rho(\al)\nabla \nonumber\\
{} = q^{2} \lop_{\nabla}(\shl \al_{1}, \be_{1}\shr)
+ qr\big(\lie_{ \al_{1}^{\ssharp}} \sd^{\ast}_{\nabla}\be_{2} - \tfrac{1}{3}(\be_{2}, \lop_{\nabla}(\al_{1})) - \lie_{\be_{1}^{\ssharp}} \sd^{\ast}_{\nabla}\al_{2} + \tfrac{1}{3}(\al_{2}, \lop_{\nabla}(\be_{1})) \big)\nonumber\\
\quad{} + qs \big(\shl \al_{1}, \be_{3} \shr + \shl \al_{3}, \be_{1}\shr\big)+ \tfrac{r^{2}}{3}\big(\shl\al_{2}, \be_{2}\shr + \sd^{\ast}_{\nabla}(\al_{2}, \be_{2})\big) + \tfrac{rs}{3}\big((\al_{2}, \be_{3}) + (\al_{3}, \be_{2})\big).\label{brho2}
\end{gather}
By~\eqref{sdasttransform}, there holds $\sd^{\ast}_{\phi_{t}^{\ast}(\nabla)}\ga - \sd^{\ast}_{\nabla}\ga = \tfrac{1}{3}(\ga, \phi_{t}^{\ast}\nabla - \nabla)$ for $\ga \in \ssec^{k}(M)$ and the local f\/low $\phi_{t}$ generated by $Z \in \symplecto(M, \Om)$. Since $\sd^{\ast}_{\phi_{t}^{\ast}(\nabla)}\ga = \phi_{t}^{\ast}(\sd^{\ast}_{\nabla} \phi_{-t}^{\ast}(\ga))$, dif\/ferentiating at $t = 0$ yields
\begin{gather}\label{liexsdast}
[\lie_{Z}, \sd^{\ast}]\ga = (\lie_{Z}\sd^{\ast})\ga = \tfrac{1}{3}(\ga, \lop(Z^{\sflat})).
\end{gather}
Applying~\eqref{liexsdast} to simplify~\eqref{brho2} yields
\begin{gather*}
\lin(\rho(\al)) \rho(\be)\nabla - \lin(\rho(\be))\rho(\al)\nabla
= q^{2} \lop_{\nabla}(\shl \al_{1}, \be_{1}\shr) + r\sd^{\ast}_{\nabla}\big( q\big( \shl \al_{1}, \be_{2} \shr + \shl \al_{2}, \be_{1} \shr\big) + \tfrac{r}{3} (\al_{2}, \be_{2})\big)\!\\
\qquad{} + qs \big(\shl \al_{1}, \be_{3} \shr + \shl \al_{3}, \be_{1}\shr\big)+ \tfrac{r^{2}}{3}\shl\al_{2}, \be_{2}\shr + \tfrac{rs}{3}\big((\al_{2}, \be_{3}) + (\al_{3}, \be_{2})\big).
\end{gather*}
Hence
\begin{gather*}
\big(\lin(\rho(\al))\rho(\be) - \lin(\rho(\be))\rho(\al) - \rho([\al, \be]_{\trun}\big)\nabla \\
{} = q(q-1) \lop_{\nabla}(\shl \al_{1}, \be_{1}\shr) + r\sd^{\ast}_{\nabla}\big( (q- 1)\left( \shl \al_{1}, \be_{2} \shr + \shl \al_{2}, \be_{1} \shr\right) + \tfrac{r-3}{3} (\al_{2}, \be_{2})\big)\\
\quad{}+ (q - 1)s \big(\shl \al_{1}, \be_{3} \shr + \shl \al_{3}, \be_{1}\shr\big)+ (\tfrac{r^{2}}{3} - s)\shl\al_{2}, \be_{2}\shr
+ \tfrac{s(r - 3)}{3}\big((\al_{2}, \be_{3}) + (\al_{3}, \be_{2})\big).
\end{gather*}
This vanishes if $q = 1$ and $r = 3 = s$. This yields~\eqref{balnabla} and shows that~\eqref{balnabla} def\/ines a symplectic af\/f\/ine action. That the stabilizer of $\nabla \in \symcon(M, \Om)$ has the form~\eqref{bsecstab} is straightforward.
\end{proof}

\begin{Corollary}\label{halnablacorollary}
Let $(M, \Om)$ be a symplectic manifold. Let $(\hhamsec(M), [\dum,\dum])$ be the Lie algebra defined by~\eqref{compatbracket} and~\eqref{hamsecdefined}. The map $\pi^{\hhamsec}\colon \hhamsec(M) \times \symcon(M, \Om) \to \symcon(M, \Om)$ defined by
\begin{gather}\label{halnabla}
\pi^{\hhamsec}(\al, \nabla) = \al \cdot \nabla = \nabla + \lop_{\nabla}(\al_{1}) + 3\sd^{\ast}_{\nabla}\al_{2} + 3\al_{3}= \nabla - \hop_{\nabla}(\al_{0}) + 3\sd^{\ast}_{\nabla}\al_{2} + 3\al_{3},
\end{gather}
for $\al \in \hhamsec(M)$ and $\nabla \in \symcon(M, \Om)$, is a symplectic affine action of $(\hhamsec(M), [\dum, \dum])$ on $(\symcon(M, \Om), \sOm)$. The stabilizer of $\nabla \in \symcon(M, \Om)$ is the Lie subalgebra
\begin{gather*}
\hhamsec(M)_{\nabla} = \big\{ \al \in \hhamsec(M)\colon 3\al_{3} = - \lop_{\nabla}(\al_{1}) - 3\sd^{\ast}_{\nabla}\al_{2}\big\}.
\end{gather*}
\end{Corollary}
\begin{proof}
By def\/inition $\pi^{\hhamsec} = \pi^{\hbsec} \circ (\pi_{1}\times \Id_{\symcon(M, \nabla)})$ where $\pi^{\hbsec}$ is def\/ined by~\eqref{balnabla}. By Lemmas~\ref{hamseccentrallemma} and~\ref{hbseclemma}, $\pi^{\hhamsec}$ is a symplectic af\/f\/ine action.
\end{proof}

\begin{Remark}
Consider the more general map $\hhamsec(M) \times \symcon(M, \Om) \to \symcon(M, \Om)$ def\/ined by
\begin{gather}\label{halnablagen}
\al \cdot \nabla = \nabla + p \hop_{\nabla}(\al_{0}) + q\lop_{\nabla}(\al_{1}) + r\sd^{\ast}_{\nabla}\al_{2} + s\al_{3},
\end{gather}
for $p, q, r, s \in \rea$. As in the proof of Lemma~\ref{hbseclemma}, it is straightforward to show that the associated map $\rho\colon (\hhamsec(M), [\dum, \dum]) \to \aff(\symcon(M, \Om))$ given by $\rho(\al)\nabla = \al \cdot \nabla - \nabla$ is a symplectic af\/f\/ine transformation of $\symcon(M, \Om)$ with linear part
\begin{gather*}
\lin(\rho(\al))\Pi = p \shl \sd^{\ast}_{\nabla}\al_{0}, \Pi\shr + q\shl \al_{1}, \Pi\shr + \tfrac{r}{3}(\al_{2}, \Pi)= (q - p)\shl \al_{1}, \Pi\shr + \tfrac{r}{3}(\al_{2}, \Pi).
\end{gather*}
Arguing as in the proof of Lemma~\ref{hbseclemma} shows directly that~\eqref{halnablagen} def\/ines a symplectic af\/f\/ine action if $q = p + 1$ and $r = 3 = s$. In this case,~\eqref{halnablagen} has the form~\eqref{halnabla}, whatever is the value of~$p$.
\end{Remark}

\section[Hamiltonian action of the extended Lie algebra]{Hamiltonian action of the extended Lie algebra\\ on $\boldsymbol{(\symcon(M, \Om),\sOm)}$}\label{symplecticaffinesection}

In this section $M$ is supposed compact. This is needed only to guarantee convergence of the integrals that appear. With appropriate qualif\/ications, everything extends to noncompact $M$, but this requires fussing that would be distracting here.

\subsection{}\label{maintheoremsection}
A reference connection $\nabla_{0} \in \symcon(M, \Om)$ and a f\/ixed $\Pi \in T_{\nabla_{0}}\symcon(M, \Om)$ determine a function $\tf_{\nabla_{0}, \Pi}(\nabla) = \laa\Pi, \nabla - \nabla_{0}\raa$ on $\symcon(M, \Om)$. Since $\vr_{\Upsilon}\tf_{\nabla_{0}, \Pi}(\nabla) = \laa \Pi, \Upsilon \raa$, the Hamiltonian vector f\/ield $\shm_{\tf(\nabla_{0}, \Pi)}= -\Pi$ is that generating the f\/low on $\symcon(M, \Om)$ given by translation by $-t\Pi$.

Given $p, q, r \in \rea$, $\nabla_{0} \in \symcon(M, \Om)$, and $\al \in \hhamsec(M)$ def\/ine $\emom^{p,q,r}_{\al}\colon \symcon(M, \Om) \to \rea$ by
\begin{gather*}
\emom^{p,q,r}_{\al}(\nabla)= p\laa \al_{0}, \K(\nabla) \raa + q\scal_{\al_{2}}(\nabla) + r\tf_{\nabla_{0}, \al_{3}}(\nabla).
\end{gather*}
The dependence of $\emom^{p, q, r}_{\al}$ on the reference connection $\nabla_{0}$ is not indicated so that the notation does not become too cluttered.
By Theorem~\ref{momentmaptheorem}, Lemma~\ref{alhamlemma}, and the preceding paragraph, the Hamiltonian vector f\/ield $\shm_{\emom^{p,q,r}_{\al}}$ generated on $\symcon(M, \Om)$ by $\emom^{p,q,r}_{\al}$ is
\begin{gather*}
\shm_{\emom^{p,q,r}_{\al}} = -p\hop(\al_{0}) - q\sd^{\ast}\al_{2} - r\al_{3}.
\end{gather*}
In the case $p = -1$, $q = 3$, and $r = 3$, there is written simply $\emom$ for $\emom^{-1, 3, 3}$. Lemma~\ref{emombracketslemma} computes the Poisson brackets $\spl \emom_{\al}, \emom_{\be}\rpl$. The following preliminary lemma is needed in its proof.

\begin{Lemma}\label{albecurvlemma}
Let $(M, \Om)$ be a compact symplectic manifold. For $\nabla = \nabla_{0} + \Pi_{ij}\,^{k},\nabla_{0} \in \symcon(M, \Om)$, $\al, \be \in \ssec^{2}(M)$, and $\al \twprod \be$ as defined in Lemma~{\rm \ref{poissonlemma}},
\begin{gather}\label{albecurv}
\laa \al \twprod \be, \curv(\nabla) \raa - 2\laa \sd_{\nabla}\al, \sd_{\nabla}\be \raa
= \laa \al \twprod \be, \curv(\nabla_{0}) \raa - 2\laa \sd_{\nabla_{0}}\al, \sd_{\nabla_{0}}\be \raa + \laa\shl \al, \be \shr, \Pi \raa.
\end{gather}
\end{Lemma}

\begin{proof}
For $\al, \be \in \ssec^{2}(M)$, $\curv(\nabla) = \curv(\nabla_{0}) + d_{\nabla_{0}}\Pi_{ijkl} + 2\Pi_{pl[i}\Pi_{j]k}\,^{p}$, by~\eqref{riemvar}, so
\begin{gather*}
\laa \al \twprod \be, \curv(\nabla) \raa - \laa \al \twprod \be, \curv(\nabla_{0}) \raa = \laa \al \twprod \be, d_{\nabla_{0}}\Pi + \Pi \wedge \Pi \raa.
\end{gather*}
Integrating by parts and using $(\sd_{\nabla_{0}}\al, \Pi) = -3(\sd_{\nabla_{0}}\al)^{p}\Pi_{ijp}$ and~\eqref{schouten} yields
\begin{gather}\label{albedn}
 \laa \al \twprod \be, d_{\nabla_{0}}\Pi \raa = -\tfrac{2}{3}\laa (\sd_{\nabla_{0}}\al, \Pi), \be\raa + \tfrac{2}{3}\laa \al, (\sd_{\nabla_{0}}\be, \Pi)\raa+ \laa \shl\al, \be\shr, \Pi\raa.
\end{gather}
For $\al, \be \in \ssec^{2}(M)$,
\begin{gather}\label{albepipi}
\al^{ab}\be^{cd}\Pi_{acp}\Pi_{bd}\,^{p} = \al^{ba}\be_{dc}\Pi_{acp}\Pi_{bd}\,^{p} = \al^{ab}\be_{cd}\Pi_{bdp}\Pi_{ac}\,^{p} = - \al^{ab}\be^{cd}\Pi_{acp}\Pi_{bd}\,^{p}
\end{gather}
so that $\al^{ab}\be^{cd}\Pi_{acp}\Pi_{bd}\,^{p} = 0$. By~\eqref{albepipi},
\begin{gather}\label{albepipi2}
\laa \al \twprod \be, \Pi \wedge \Pi \raa = \tfrac{1}{2}\int_{M}(\al \twprod \be)_{ijkl}(\Pi \wedge \Pi)^{ijkl} \, \Om_{n} = 2\int\al^{ab}\Pi_{abp}\be^{cd}\Pi_{cd}\,^{p}\,\Om_{n}.
\end{gather}
From
\begin{gather*}
\dn\al_{[ij]i_{1}\dots i_{k-1}}= d_{\nabla_{0}}\al_{[ij]i_{1}\dots, i_{k-1}}
-(k-1)\big(\Pi_{i(i_{1}}\,^{p}\al_{i_{2}\dots i_{k-1})jp} - \Pi_{j(i_{1}}\,^{p}\al_{i_{2}\dots i_{k-1})ip}\big),
\end{gather*}
there follows
\begin{gather}\label{sdtransform}
\sd_{\nabla}\al_{i_{1}\dots i_{k-1}} = \sd_{\nabla_{0}}\al_{i_{1}\dots i_{k-1}} + (-1)^{k-1}(k-1)\Pi^{pq}\,_{(i_{1}}\al_{i_{2}\dots i_{k-1})pq}.
\end{gather}
Computing using~\eqref{sdtransform}, $(\sd_{\nabla_{0}}\al, \Pi) = -3(\sd_{\nabla_{0}}\al)^{p}\Pi_{ijp}$, and~\eqref{albepipi2} yields
\begin{gather}
\laa \sd_{\nabla}\al, \sd_{\nabla}\be \raa - \laa \sd_{\nabla_{0}}\al, \sd_{\nabla_{0}}\be\raa\nonumber\\
\qquad{}= -\tfrac{1}{3}\laa (\sd_{\nabla_{0}}\al, \Pi), \be\raa + \tfrac{1}{3}\laa \al, (\sd_{\nabla_{0}}\be, \Pi)\raa+ \tfrac{1}{2}\laa \al\twprod \be, \Pi \wedge \Pi\raa.\label{albesdsd}
\end{gather}
Combining~\eqref{albedn} and~\eqref{albesdsd} yields~\eqref{albecurv}.
\end{proof}

\begin{Lemma}\label{emombracketslemma}
Let $(M, \Om)$ be a compact symplectic manifold and fix $\nabla_{0} \in \symcon(M, \Om)$. Define a~skew-symmetric map $\mcoc_{\nabla_{0}}\colon \hhamsec(M) \times \hhamsec(M) \to \rea$ by
\begin{gather}
\mcoc_{\nabla_{0}}(\al, \be)  = 3 \laa \al_{2}\twprod \be_{2}, \curv(\nabla_{0}) \raa - 6\laa \sd_{\nabla_{0}} \al_{2}, \sd_{\nabla_{0}}\be_{2}\raa
\nonumber\\
\hphantom{\mcoc_{\nabla_{0}}(\al, \be)  =}{}
 - 3\sOm( \shm_{\emom_{\al}}(\nabla_{0}), \be_{3}) - 3\sOm( \al_{3}, \shm_{\emom_{\be}}(\nabla_{0}))- 9 \laa \al_{3}, \be_{3}\raa.\label{mcocdefined}
\end{gather}
For $\al, \be \in \hhamsec(M)$, the Poisson brackets on $(\symcon(M, \Om), \sOm)$ of $\emom_{\al}$ and $\emom_{\be}$ satisfy
\begin{gather*}
\spl \emom_{\al}, \emom_{\be}\rpl - \emom_{[\al, \be]} = \mcoc_{\nabla_{0}}(\al, \be),
\end{gather*}
where $\mcoc_{\nabla_{0}}(\al, \be)$ is regarded as a constant function on $\symcon(M, \Om)$.
\end{Lemma}

\begin{proof}
First, there are made some general observations needed later in the proof. Dif\/fe\-ren\-tia\-ting $\K(\phi_{t}^{\ast}\nabla) = \K(\nabla) \circ \phi_{t}$ along the f\/low $\phi_{t}$ of $X \in \symplecto(M, \Om)$, and using~\eqref{kvary} yields
\begin{gather}\label{isotropyidentity}
\hop_{\nabla}^{\ast}\hop_{\nabla}(f) = \{f, \K(\nabla)\},\qquad
\laa \hop_{\nabla}(f), \hop_{\nabla}(g)\raa = \laa f, \{g, \K(\nabla)\}\raa = \laa \{f, g\}, \K(\nabla)\raa ,
\end{gather}
for $f, g, \in \cinf(M)$. Similarly, for $f \in \cinf(M)$ and $\al \in \ssec^{2}(M)$, by~\eqref{tracedhf},
\begin{align}
\laa \hop_{\nabla}(f), \sd^{\ast}_{\nabla}\al\raa &= - \laa \sd_{\nabla}\hop_{\nabla}(f), \al\raa = - \laa \lie_{\hm_{f}}\ric(\nabla), \al\raa\nonumber\\
& = \laa \ric(\nabla), \lie_{\hm_{f}}\al \raa = \laa \ric(\nabla), \shl \sd^{\ast}f, \al \shr\raa = \scal_{\shl \sd^{\ast}f, \al \shr}(\nabla).\label{ii2}
\end{align}
By def\/inition of $\hop_{\nabla}$, for $f \in \cinf(M)$ and $\al \in \ssec^{3}(M)$,
\begin{gather}\label{ii3}
\laa \hop_{\nabla}(f), \al\raa =\laa \hop_{\nabla_{0}}(f), \al\raa + \laa \shl \sd^{\ast}f, \al\shr, \nabla - \nabla_{0}\raa.
\end{gather}
By~\eqref{sdasttransform}, for $\al_{2}, \be_{2} \in \ssec^{2}(M)$ and $\al_{3}, \be_{3} \in \ssec^{3}(M)$,
\begin{gather}
\laa \sd_{\nabla}^{\ast}\al_{2}, \be_{3} \raa + \laa \al_{3}, \sd_{\nabla}^{\ast}\be_{2}\raa\nonumber\\
 \qquad{} = \laa \sd_{\nabla_{0}}^{\ast}\al_{2}, \be_{3} \raa + \laa \al_{3}, \sd_{\nabla_{0}}^{\ast}\be_{2}\raa + \tfrac{1}{3}\laa (\al_{2}, \be_{3}) + (\al_{3}, \be_{2}), \nabla - \nabla_{0}\raa.\label{albe23}
\end{gather}
For $\al, \be \in \hhamsec(M)$,
\begin{gather}
\spl \emom_{\al}^{p, q, r}, \emom_{\be}^{p, q, r} \rpl (\nabla) = \sOm_{\nabla}(\shm_{ \emom_{\al}^{p, q, r}}, \shm_{ \emom_{\be}^{p, q, r}})\nonumber\\
{} = p^{2}\laa \hop_{\nabla}(\al_{0}), \hop_{\nabla}(\be_{0})\raa + \tfrac{q^{2}}{3}\big( \laa \al_{2} \twprod \be_{2}, \curv(\nabla)\raa -2\laa \sd_{\nabla} \al_{2}, \sd_{\nabla}\be_{2}\raa+ \scal_{(\al_{2}, \be_{2})}(\nabla)\big)\nonumber\\
\quad{} + r^{2} \laa \al_{3}, \be_{3}\raa+ pq\big( \laa \hop_{\nabla}(\al_{0}), \sd^{\ast}_{\nabla}\be_{2}\raa + \laa \sd_{\nabla}^{\ast}\al_{2}, \hop_{\nabla}(\be_{0}) \raa\big) \nonumber\\
\quad {}+ pr\left( \laa \hop_{\nabla}(\al_{0}), \be_{3}\raa + \laa \al_{3}, \hop_{\nabla}(\be_{0}) \raa \right) + qr\big( \laa \sd_{\nabla}^{\ast}\al_{2}, \be_{3} \raa + \laa \al_{3}, \sd^{\ast}_{\nabla}\be_{2}\raa \big).\label{malmbe1}
\end{gather}
Simplifying~\eqref{malmbe1} using~\eqref{isotropyidentity},~\eqref{ii2},~\eqref{ii3},~\eqref{albecurv} of Lemma~\ref{albecurvlemma},~\eqref{albe23}, and Lemma~\ref{poissonlemma} yields
\begin{gather*}
\spl \emom_{\al}^{p, q, r}, \emom_{\be}^{p, q, r} \rpl (\nabla)
 = p^{2}\laa \{\al_{0}, \be_{0}\}, \K(\nabla) \raa - pq\scal_{\shl \al_{1}, \be_{2}\shr + \shl \al_{2}, \be_{1}\shr}(\nabla)+ \tfrac{q^{2}}{3}\scal_{(\al_{2}, \be_{2})}(\nabla) \\
\qquad{}- pr\tf_{\nabla_{0}, \shl\al_{1}, \be_{3}\shr + \shl \al_{3}, \be_{1}\shr}( \nabla ) + \tfrac{qr}{3}\tf_{\nabla_{0}, (\al_{2}, \be_{3}) + (\al_{3}, \be_{2})}(\nabla) + \tfrac{q^{2}}{3}\tf_{\nabla_{0}, \shl \al_{2}, \be_{2}\shr}(\nabla)\\
\qquad{} + r^{2} \laa \al_{3}, \be_{3}\raa+ \tfrac{q^{2}}{3}\big( \laa \al_{2}\twprod \be_{2}, \curv(\nabla_{0}) \raa -2\laa \sd_{\nabla_{0}} \al_{2}, \sd_{\nabla_{0}}\be_{2}\raa\big) \\
\qquad{} + pr\big( \laa \hop_{\nabla_{0}}(\al_{0}), \be_{3}\raa + \laa \al_{3}, \hop_{\nabla_{0}}(\be_{0}) \raa \big)
+ qr\big( \laa \sd_{\nabla_{0}}^{\ast}\al_{2}, \be_{3} \raa + \laa \al_{3}, \sd_{\nabla_{0}}^{\ast}\be_{2}\raa\big).
\end{gather*}
If $p = -1$, $q = 3$, and $r = 3$, then, writing $\emom^{-1, 3, 3} = \emom$,
\begin{gather*}
\spl \emom_{\al}, \emom_{\be}\rpl (\nabla) = \laa \{\al_{0}, \be_{0}\}, \K(\nabla) \raa + 3\scal_{\shl \al_{1}, \be_{2}\shr + \shl \al_{2},\be_{1}\shr+ (\al_{2}, \be_{2})}(\nabla) \\
\hphantom{\spl \emom_{\al}, \emom_{\be}\rpl (\nabla) =}{} +3\tf_{\nabla_{0}, \shl\al_{1}, \be_{3}\shr + \shl \al_{3}, \be_{1}\shr+ (\al_{2}, \be_{3}) + (\al_{3}, \be_{2})+ \shl \al_{2}, \be_{2}\shr}(\nabla)\\
\hphantom{\spl \emom_{\al}, \emom_{\be}\rpl (\nabla) =}{} + 9 \laa \al_{3}, \be_{3}\raa + 3\big( \laa \al_{2}\twprod \be_{2}, \curv(\nabla_{0}) \raa -2\laa \sd_{\nabla_{0}} \al_{2}, \sd_{\nabla_{0}}\be_{2}\raa\big) \\
\hphantom{\spl \emom_{\al}, \emom_{\be}\rpl (\nabla) =}{} - 3\big( \laa \hop_{\nabla_{0}}(\al_{0}), \be_{3}\raa + \laa \al_{3}, \hop_{\nabla_{0}}(\be_{0}) \raa \big) + 9\big( \laa \sd_{\nabla_{0}}^{\ast}\al_{2}, \be_{3} \raa + \laa \al_{3}, \sd_{\nabla_{0}}^{\ast}\be_{2}\raa\big)\\
\hphantom{\spl \emom_{\al}, \emom_{\be}\rpl (\nabla)}{} = \emom_{[\al, \be]}(\nabla) + \mcoc_{\nabla_{0}}(\al, \be).\tag*{\qed}
\end{gather*}\renewcommand{\qed}{}
\end{proof}

Since $\emom$ is not equivariant, the skew-symmetric map $\mcoc_{\nabla_{0}}\colon \hhamsec(M) \times \hhamsec(M) \to \rea$ is a $2$-cocycle, the \textit{nonequivariance cocycle}, for the Lie algebra cohomology of $(\hhamsec(M), [\dum, \dum])$ with coef\/f\/icients in the trivial module $\rea$. The nonequivariance cocycles associated with dif\/ferent choices of reference connection $\nabla_{0}\in \symcon(M, \Om)$ are cohomologous.
In what follows the dependence on the choice of ${\nabla_{0}}$ is sometimes suppressed, and there is written simply $\mcoc$ instead of $\mcoc_{\nabla_{0}}$.
In a f\/inite-dimensional setting it follows from the def\/inition that the nonequivariance cocycle of a weakly Hamiltonian action is a constant function. In the inf\/inite-dimensional setting here, the same conclusion would follow from the def\/inition~\eqref{mcocdefined}, had adequate analytic foundations been laid. Although they have not, Lemma~\ref{emombracketslemma} shows explicitly that $\spl\emom_{\al},\emom_{\be}\rpl - \emom_{[\al, \be]}$ is constant on $\symcon(M, \Om)$.

\subsection{}
Because $\hzsec_{i}(M)= \hzsec(M) \cap \hssec_{i}(M)= \hssec_{i}(M)$ is an ideal in $(\hzsec_{1}(M), [\dum, \dum]_{\trun})$ for any $i \geq 2$, the Lie bracket $[\dum, \dum]_{\trun}$ induces a Lie bracket (that will also be denoted $[\dum, \dum]_{\trun}$) on the vector space $\hzsec_{1}(M)/\hzsec_{4}(M)$.
As a vector space $\hzsec_{1}(M)/\hzsec_{4}(M)$ is identif\/ied with $\symplecto(M, \Om)^{\sflat} \oplus \ssec^{2}(M) \oplus \ssec^{3}(M)$, and the Lie bracket is given explicitly by
\begin{gather}
[\al, \be]_{\trun} = [ \al_{1}^{\ssharp}, \be_{1}^{\ssharp}]^{\sflat} + (\al_{2}, \be_{2}) + \lie_{\al_{1}^{\ssharp}}\be_{2} - \lie_{\be_{1}^{\ssharp}}\al_{2}\nonumber\\
\hphantom{[\al, \be]_{\trun} =}{} + (\al_{2}, \be_{3}) + (\al_{3}, \be_{2}) + \shl \al_{2}, \be_{2}\shr + \lie_{\al_{1}^{\ssharp}}\be_{3} - \lie_{\be_{1}^{\ssharp}} \al_{3} .\label{trbracketdefined}
\end{gather}
Although~\eqref{trbracketdefined} satisf\/ies the Jacobi identity by construction, this can also be checked by tedious computations using the compatibility and $\symplecto(M, \Om)$-equivariance of $(\dum, \dum)$ and $\shl\dum, \dum\shr$.

Since $\hbsec_{4}(M) = \hbsec(M) \cap \ssec_{4}(M)$ satisf\/ies $\hbsec_{4}(M) = \hzsec_{4}(M)$ and $\hbsec_{1}(M)$ is a subalgebra of $(\hzsec_{1}(M), [\dum, \dum]_{\trun})$, $(\hbsec_{1}(M)/\hbsec_{4}(M), [\dum, \dum]_{\trun})$ is a Lie subalgebra of $(\hzsec_{1}(M)/\hzsec_{4}(M), [\dum, \dum]_{\trun})$ and the exact sequence~\eqref{hamsplit2} of Lemma~\ref{hamseccentrallemma} descends to give an exact sequence
\begin{gather*}
\{0\} \longrightarrow \rea \stackrel{\iota}{\longrightarrow} \big(\hhamsec(M)/\hhamsec_{4}(M), [\dum, \dum]\big) \stackrel{\pi_{1}}{\longrightarrow} \big(\hbsec_{1}(M)/\hbsec_{4}(M), [\dum, \dum]_{\trun}\big) \longrightarrow \{0\}
\end{gather*}
of Lie algebras, where $\iota$ is the restriction to the constant functions $\rea \subset \cinf(M)$ of the map $\iota\colon \cinf(M) \to \hhamsec(M)$ def\/ined by $\iota(f) = -f + \sd^{\ast}f$ (see Lemma~\ref{hamseclemma}), and embeds $\rea$ in $\hhamsec^{0}(M) = \cinf(M)$ as the constant functions (but with a sign change). The bracket $[\dum, \dum]$ descends to a Lie bracket, also denoted $[\dum, \dum]$ on the quotient $\hhamsec(M)/(\iota(\rea) \oplus \hhamsec_{4}(M))$ by the Lie ideal $\iota(\rea) \oplus \hhamsec_{4}(M)$.

For a compact symplectic manifold $(M, \Om)$, the linear map $\con\colon (\hhamsec(M), [\dum, \dum]) \to \rea$ def\/ined by
\begin{gather*}
\con(\al) = \vol_{\Om_{n}}(M)^{-1}\int_{M}\al_{0}\,\Om_{n} = \vol_{\Om_{n}}(M)^{-1}\laa \al_{0}, 1\raa,
\end{gather*}
is a Lie algebra homomorphism because, by~\eqref{hamsec1}, $[\al, \be]_{0} = -\{\al_{0}, \be_{0}\}$, and the integral of a Poisson bracket vanishes. Def\/ine $\nu\colon \hhamsec(M) \to \hhamsec(M)$ by $\nu(\al) = \al + \iota(\con(\al))$ (so $\nu(\al)_{0} = \al_{0} - \con(\al_{0})$).
Because $\iota(\rea)$ is central in $(\hhamsec(M), [\dum, \dum])$, $\nu$ is a Lie algebra homomorphism of $(\hhamsec(M), [\dum, \dum])$. Because $\nu \circ \nu = \nu$, $\pi_{1} \circ \nu = \pi_{1}$, and $\ker \nu = \iota(\rea)$, it follows that the maps~$\nu$ and~$\pi_{1}$ induce Lie algebra isomorphisms
\begin{gather*}
 \big(\hhamsec(M)/\big(\iota(\rea) \oplus \hhamsec_{4}(M)\big), [\dum, \dum]\big) \simeq \big(\nu\big(\hhamsec(M)\big)/\hhamsec_{4}(M), [\dum, \dum]\big)
\simeq \big(\hbsec_{1}(M)/\hbsec_{4}(M), [\dum, \dum]_{\trun}\big).
\end{gather*}
When $M$ is compact, the Lie algebra $\ham(M, \Om)$ is isomorphic with the mean zero elements of~$\cinf(M)$ equipped with the Poisson bracket. It follows from Lemma~\ref{hamseclemma} that the map $\hhamsec(M) \to \hge = \ham(M, \Om) \oplus \ssec^{2}(M) \oplus \ssec^{3}(M)$ def\/ined by $\al \to -\al_{0}+ \al_{2}+ \al_{3}$ induces a linear isomorphism $\hhamsec(M)/(\iota(\rea) \oplus \hhamsec_{4}(M)) \to \hge$ with inverse $\j\colon \hge \to \hhamsec(M)/(\iota(\rea) \oplus \hhamsec_{4}(M))$ given by $\j(f + \al_{2} + \al_{3}) = \iota(f) + \al_{2} + \al_{3} = -f + \sd^{\ast}f + \al_{2} + \al_{3} \in \hhamsec(M)/(\iota(\rea) \oplus \hhamsec_{4}(M))$, and such that the pullback via $\j$ of the Lie bracket $[\dum, \dum]$, also to be written $[\dum, \dum]$, agrees, when restricted to the subspace $\ham(M, \Om)$, with the Poisson bracket.

Since for the symplectic af\/f\/ine action~\eqref{halnabla} the action of $\rea \oplus \hhamsec_{4}(M)$ is trivial, the action~\eqref{halnabla} descends to a symplectic af\/f\/ine action of $(\hhamsec(M)/(\iota(\rea) \oplus \hhamsec_{4}(M)), [\dum, \dum])$ on $\symcon(M, \Om)$. Composing with $\j$ yields a symplectic af\/f\/ine action of $(\hge, [\dum, \dum])$ on $\symcon(M, \Om)$. Precisely, the symplectic af\/f\/ine action $\pi^{\hge}$ of $\al \in \hge$ on $\symcon(M, \Om)$ is def\/ined to be $\pi^{\hge}(\al, \nabla) = \pi^{\hhamsec}(\j(\al), \nabla)$; this makes sense because~$\hhamsec_{4}(M)$ acts trivially on $\symcon(M, \Om)$. Lemma~\ref{hamiltonianlemma} summarizes the preceding as showing that this action is weakly Hamiltonian. In the statement, the dual $\hge^{\ast}$ is identif\/ied with $\hge$ via the pairing $\laa \dum, \dum \raa$. Precisely, the pairing $\laa\dum, \dum\raa$ extends to elements $\al$ and $\be$ of $\hssec(M)$ (or of $\hhamsec(M)$ or $\hge$) by linearity as $\laa \al, \be \raa = \sum_{k \geq 0}\laa \al_{k}, \be_{k}\raa$. Note that, if $M$ is compact and $\al, \be \in \hhamsec(M)$, then $\laa \al_{1}, \be_{1}\raa = \laa \sd^{\ast}\al_{0}, \sd^{\ast}\be_{0} \raa = \int_{M}\{\al_{0}, \be_{0}\}\,\Om_{n} = 0$. Since $M$ is compact elements of the linear duals of any of $\hhamsec(M)$, $\hge$, etc. can be identif\/ied with elements of the space itself via the pairing $\laa \dum, \dum \raa$.

\begin{Lemma}\label{hamiltonianlemma}
Let $(M, \Om)$ be a compact symplectic manifold. The action $\pi^{\hge}$ of $(\hge, [\dum, \dum])$, on $\symcon(M, \Om)$ induced by the symplectic affine action~\eqref{halnabla} is weakly Hamiltonian with $($nonequiva\-riant$)$ moment map $\mom\colon \symcon(M, \Om)\to \hge^{\ast}$ given by
\begin{gather*}
\laa \mom(\nabla), \al \raa = \emom_{\j(\al)}(\nabla) = \laa \al_{0}, \K(\nabla) \raa + 3\scal_{\al_{2}}(\nabla) + 3\tf_{\nabla_{0}, \al_{3}}(\nabla),
\end{gather*}
for any fixed $\nabla_{0} \in \symcon(M, \Om)$ and any $\al = \al_{0} + \al_{2} + \al_{3} \in \hge$.
\end{Lemma}

\begin{proof}
By def\/inition of the symplectic af\/f\/ine action $\pi^{\hhamsec}$, for $\al \in \hhamsec(M)$,
\begin{gather*}
\shm_{\emom_{\al}}(\nabla) = \hop_{\nabla}(\al_{0}) - 3\sd_{\nabla}^{\ast}\al_{2} - 3\al_{3} = \tfrac{d}{dt}\big|_{t = 0}\pi^{\hhamsec}(-t\al, \nabla).
\end{gather*}
This shows that $\shm_{\emom_{\al}}$ is the vector f\/ield generated by the action of $\al$ on $\symcon(M, \Om)$, so $\emom(\nabla)$ is a (nonequivariant) moment map for the action of $\hhamsec(M)$ on $\symcon(M, \Om)$, and this action is weakly Hamiltonian. The claim follows formally from the def\/initions of $(\hge, [\dum, \dum])$ and $\pi^{\hge}$.
\end{proof}

Fix $\nabla_{0} \in \symcon(M, \Om)$. Def\/ine $\sag = \cinf(M) \oplus \ssec^{2}(M) \oplus \ssec^{3}(M)$ and extend the pairing $\laa \dum, \dum \raa$ to~$\sag$ as for~$\hge$. The linear isomorphism $\j\colon \hge \to \hhamsec(M)/(\iota(\rea) \oplus \hhamsec_{4}(M))$ lifts to a linear isomorphism $\hj\colon \sag \to \hhamsec(M)/\hhamsec_{4}(M)$ (def\/ined in the same way as $\j$, by $\hj(f + \al_{2} + \al_{3}) = \iota(f) + \al_{2} + \al_{3} \in \hhamsec(M)/\hhamsec_{4}(M)$). As for $\j$, the pullback to $\sag$ via $\hj$ of the Lie bracket $[\dum, \dum]$, agains also to be written $[\dum, \dum]$, agrees, when restricted to the subspace $\cinf(M)$, with the Poisson bracket. Since~$\mcoc$ vanishes when restricted to $\hhamsec_{4}(M)$ in either argument, it can be viewed as a pairing on $\hhamsec(M)/\hhamsec_{4}(M)$, so it makes sense to def\/ine a Lie bracket $\jl\dum, \dum\jr$ on $\sag$ by
\begin{gather}\label{stbracketdefined}
\jl\al, \be \jr = [\al, \be] + \mcoc\big(\hj(\al), \hj(\be)\big),
\end{gather}
Here $\mcoc$ is regarded as taking values in the subspace of constant functions $\rea \subset \cinf(M)$. Since $\iota(\rea) = \ker \nu$, $\nu$ descends from $\hhamsec(M)$ to give a Lie algebra homomorphism
\begin{gather*}
\nu\colon \ \big(\hhamsec(M)/\hhamsec_{4}(M), [\dum,\dum]\big)\to \big(\hhamsec(M)/(\iota(\rea) \oplus \hhamsec_{4}(M)), [\dum,\dum]\big).
\end{gather*}
The linear map $\bar{\nu}\colon \hj \to \j$ def\/ined by $\hj(f + \al_{2} + \al_{3}) = f - \con(f) + \al_{2} + \al_{3}$ satisf\/ies $\j \circ \bar{\nu} = \nu \circ \hj$, so is a Lie algebra homomorphism
$\bar{\nu}\colon (\sag, \jl\dum, \dum\jr)\to(\hge, [\dum, \dum])$. In what follows it is convenient, and should cause no confusion, to omit the bar and write simply $\nu$ in place of $\bar{\nu}$.
The Lie algebra $(\sag, \jl\dum, \dum\jr)$ is the central extension of $(\hge, [\dum, \dum])$ corresponding to the $2$-cocycle $\mcoc$. Precisely, the following sequence of Lie algebras is exact:
\begin{gather*}
\{0\} \longrightarrow \rea \longrightarrow \big(\sag, \jl\dum, \dum\jr\big)\stackrel{\nu}{\longrightarrow}(\hge, [\dum, \dum]) \longrightarrow \{0\}.
\end{gather*}
The action $\hat{\pi}^{\sag}(\al, \nabla)$ of $(\sag, \jl\dum,\dum\jr)$ on $\symcon(M, \Om)$ is def\/ined to be the action $\pi^{\hge}(\nu(\al), \nabla)$ of its image via the projection $\nu$, so that the center $\rea$ of $(\sag, \jl\dum, \dum\jr)$ acts trivially.

\begin{Theorem}\label{emomtheorem}
On a compact symplectic manifold $(M, \Om)$, for any $\nabla_{0} \in \symcon(M, \Om)$ the action of $(\sag, \jl\dum, \dum\jr)$ on $(\symcon(M, \Om), \sOm)$ is Hamiltonian, with equivariant moment map $\hemom\colon \symcon(M, \Om) \to \sag^{\ast}$ given by
\begin{gather*}
\laa\hemom(\nabla), \al \raa = \laa \mom(\nabla), \nu(\al) \raa - \con(\hj(\al)) = \emom_{\j \circ \nu(\al)}(\nabla) - \con(\hj(\al))= \emom_{\j \circ \nu(\al)}(\nabla) +\con(\al_{0})
\end{gather*}
for any $\al = \al_{0} + \al_{2} + \al_{3} \in \sag$, where $\mom$ is determined by $\nabla_{0}$ as in Lemma~{\rm \ref{hamiltonianlemma}}.
\end{Theorem}

\begin{proof}
Let $\al, \be \in \sag$. It is convenient to write $\hemom_{\al}$ and $\mom_{\nu(\al)}$ in place of $\laa\hemom(\nabla), \al \raa$ and $\laa \mom(\nabla), \nu(\al) \raa$.

By~\eqref{intk}, when $M$ is compact, the constant $\kappa(M, \Om) = \int_{M}\K(\nabla)\,\Om_{n}$ does not depend on the choice of $\nabla \in \symcon(M, \Om)$. For $\be \in \hhamsec(M)$,
\begin{gather*}
\emom_{\nu(\be)}(\nabla) - \emom_{\be}(\nabla) = \laa \con(\be), \K(\nabla)\raa = \con(\be)\kappa(M, \Om)
\end{gather*}
is a constant function on $\symcon(M, \Om)$. It follows that $\hemom_{\al}$, $\mom_{\nu(\al)}$, and $\emom_{\j \circ \nu(\al)}$ dif\/fer by constant functions, so they generate on $\symcon(M, \Om)$ the same Hamiltonian vector f\/ield. This is the vector f\/ield generated by the symplectic af\/f\/ine action of $\nu(\al)$ on $\symcon(M, \Om)$, and this equals the lifted action of $\al$ on $\symcon(M, \Om)$. This suf\/f\/ices to show that $\hemom$ is a moment map for the action of $(\sag, \jl\dum,\dum\jr)$ on $\symcon(M, \Om)$, provided the equivariance of $\hemom$ can be shown.

For $\al, \be\in \hhamsec(M)$, since, in the def\/inition~\eqref{mcocdefined} of $\mcoc(\al, \be)$, there enter $\al_{0}$ and $\be_{0}$ only via $\shm_{\emom_{\al}}$ and $\shm_{\emom_{\be}}$, and, by the preceding remarks, these vector f\/ields equal $\shm_{\emom_{\nu(\al)}}$ and $\shm_{\emom_{\nu(\be)}}$, there holds $\mcoc(\nu(\al), \nu(\be)) = \mcoc(\al, \be)$.

For $\al, \be \in \sag$, because $[\hhamsec(M), \hhamsec(M)] \subset \ker \con$,
\begin{align}
\con(\hj \jl \al, \be\jr) &= \con(\hj [ \al, \be]) + \con(\hj (\mcoc( \hj(\al), \hj(\be))))\nonumber \\
& = \con([\hj( \al), \hj(\be)]) - \mcoc( \hj(\al), \hj(\be)) = - \mcoc( \hj(\al), \hj(\be)).\label{concoc}
\end{align}
By the def\/initions of $\hemom$ and $\mcoc$, for $\al, \be \in \sag$,
\begin{align*}
\spl \hemom_{\al}, \hemom_{\be}\rpl & = \spl \mom_{\nu(\al)}, \mom_{\nu(\be)}\rpl = \spl \emom_{\j \circ\nu(\al)}, \emom_{\j \circ\nu(\be)}\rpl \\
& = \emom_{[\j \circ\nu(\al), \j \circ \nu(\be)]} + \mcoc(\j \circ \nu(\al), \j \circ \nu(\be))
 = \emom_{\j ([\nu(\al), \nu(\be)])} + \mcoc(\nu \circ \hj (\al), \nu \circ \hj(\be))\\
& = \emom_{\j \circ \nu(\jl\al , \be\jr)} + \mcoc(\hj(\al), \hj(\be))= \emom_{\j \circ \nu(\jl\al , \be\jr)} - \con(\hj \jl \al, \be\jr) = \hemom_{\jl \al, \be \jr},
\end{align*}
where the penultimate equality follows from~\eqref{concoc}.
\end{proof}

Viewed as a map taking values in the dual of $\hhamsec(M)/(\iota(\rea)\oplus \hhamsec_{4}(M))$, the map $\mom(\nabla)$ is identif\/ied via $\laa \dum, \dum \raa$ with the element
\begin{gather*}
\j\left(\K(\nabla) + 3\ric(\nabla) + 3(\nabla - \nabla_{0})\right)= -\K(\nabla) + \sd^{\ast}\K(\nabla) + 3\ric(\nabla) + 3(\nabla - \nabla_{0})
\end{gather*}
of $\hhamsec(M)/(\iota(\rea)\oplus \hhamsec_{4}(M))$, so the map $\mom(\nabla) \in \hge^{\ast}$ is identif\/ied via $\laa\dum, \dum \raa$ with
\begin{gather*}
\K(\nabla) + 3\ric(\nabla) + 3(\nabla - \nabla_{0}) \in \hge.
\end{gather*}
Similarly $\hemom(\nabla) \in \sag^{\ast}$ is identif\/ied with
\begin{gather}\label{hemomfull}
\hemom(\nabla)= \vol_{\Om_{n}}(M)^{-1} + \bar{\K}(\nabla) + 3\ric(\nabla) + 3(\nabla - \nabla_{0}) \in \sag,
\end{gather}
where $\sag^{\ast}$ is identif\/ied with $\sag$ using $\laa \dum, \dum \raa$, $\bar{\K}(\nabla) = \K(\nabla) - \ka(M, \Om)$ is the mean zero part of $\K(\nabla)$, and the constant function $\vol_{\Om_{n}}(M)^{-1}$ is identif\/ied via $\laa \dum, \dum \raa$ with the map $f \in \cinf(M) \to \con(f)$.

On a compact symplectic manifold f\/ix $\nabla_{0} \in \symcon(M, \Om)$ with corresponding moment map $\hemom\colon \symcon(M, \Om) \to \sag^{\ast}$ and def\/ine $\henon\colon \symcon(M, \Om)\to \rea$ by
\begin{gather*}
\henon(\nabla) = \laa \hemom(\nabla), \hemom(\nabla) \raa =\vol_{\Om_{n}}(M)^{-1} \big(1- \kappa(M, \Om)^{2}\big) + \emf(\nabla) + 18\rc_{(1)}(\nabla),
\end{gather*}
where $\emf(\nabla) = \int_{M}\K(\nabla)^{2}\Om_{n}$ and $\rc_{(1)}$ is def\/ined in~\eqref{rckdefined}. By~\eqref{varex1} and~\eqref{varrck}, the f\/irst variation of $\henon$ is
\begin{gather*}
\vr_{\Pi}\henon(\nabla) = 2\sOm_{\nabla}(\hop(\K) + 9\sd^{\ast}\ric, \Pi),
\end{gather*}
so the critical points of $\henon$ are the solutions of $\hop(\K) + 9\sd^{\ast}\ric = 0$.

\subsection{}
For $s, t \in \reat$ let $\Psi_{s, t} \in \eno(\hssec(M))$ be the linear automorphism def\/ined in Lemma~\ref{lieisolemma}. Note that, although $\Psi_{s, t}$ does preserve $\hbsec_{1}(M)$, it does not preserve $\hhamsec(M)$. Rather, the subspace $\hhamsec^{s, t}(M) = \{\al \in \hhamsec(M)\colon s\al_{1} =-t\sd^{\ast}\al_{0}\}$ is a subalgebra of $(\hssec(M), [\dum, \dum]_{s, t})$ and $\Psi_{s, t}$ maps it isomorphically onto $\hhamsec(M)$. Hence, composing $\pi^{\hhamsec}$ with $\Psi_{s, t}\times\Id_{\symcon(M, \Om)}$ yields a degenerated symplectic af\/f\/ine action of $(\hhamsec^{s, t}(M), [\dum, \dum]_{s, t})$ on $\symcon(M, \Om)$ given by
\begin{align*}
\pi^{\hhamsec}_{s, t}(\al, \nabla) &= \nabla +t\lop_{\nabla}(\al_{1}) + 3s\sd^{\ast}_{\nabla}\al_{2} + 3t^{-1}s^{2}\al_{3}\\
&= \nabla - s^{-1}t^{2}\hop_{\nabla}(\al_{0}) + 3s\sd^{\ast}_{\nabla}\al_{2} + 3t^{-1}s^{2}\al_{3}.
\end{align*}
These actions are all isomorphic as long as $s$ and $t$ are both nonzero. The extremal case $(s, t) = (0, 1)$ yields a symplectic af\/f\/ine action of $(\hhamsec(M), \shl \dum, \dum \shr)$ on $\symcon(M, \Om)$ that corresponds up to a constant factor with the usual action of $\ham(M, \Om)$ on $\symcon(M, \Om)$. All the constructions made starting from $(\hhamsec(M), [\dum, \dum])$ can be repeated with $(\hhamsec^{s, t}(M), [\dum, \dum]_{s, t})$ in its place (this requires modifying the maps $\j$ and $\hj$). There results a moment map $\hemom^{s, t}$ for a Lie bracket $\jl\dum, \dum \jr_{s, t}$ on $\sag$.

However, it is easier, and equivalent, to introduce the parameters $(s, t)$ at the end of the construction, by viewing $\Psi_{s, t}$ as a linear endomorphism of $\sag$, and def\/ining $\jl\dum, \dum \jr_{s, t}$ directly to be the pullback of $\jl\dum, \dum \jr$ via $\Psi_{s, t}$. Here there has to be made precise how to view $\Psi_{s, t}$ as an endomorphism of $\sag$. It is simpler still to work with a modif\/ied map $\psi_{s, t}\colon \sag \to \sag$ def\/ined by $\psi_{s, t}(\al_{0} + \al_{2} + \al_{3}) = t\al_{0} + s\al_{2} + t^{-1}s^{2} \al_{3}$, and to def\/ine $\jl\dum, \dum \jr_{s, t}$ to be the pullback of $\jl\dum, \dum \jr$ via $\psi_{s, t}$
The composed action $\pi^{\sag}\circ (\psi_{s, t} \times \Id_{\sag})$ of $(\sag, \jl\dum, \dum \jr_{s, t})$ is symplectic af\/f\/ine, and the corresponding moment map $\hemom^{s, t}\colon \symcon(M, \Om) \to \sag^{\ast}$ is $ \hemom^{s, t} = \psi_{s, t}^{\ast} \circ\hemom$ where $\psi_{s, t}^{\ast} \in \eno(\sag^{\ast})$ is the linear dual of $\psi_{s, t}$.
The corresponding quadratic functional $\henon_{s, t}(\nabla)$ is
\begin{align}
\henon_{s, t}(\nabla)& = \laa \hemom^{s, t, u}(\nabla), \hemom^{s, t, u}(\nabla) \raa \nonumber\\
&= t^{2}\vol_{\Om_{n}}(M)^{-1}\big(1 - \kappa(M, \Om)^{2} + t^{2}\emf(\nabla) + 18s^{2}\rc_{(1)}(\nabla)\big).\label{engstdefined}
\end{align}
By~\eqref{varex1} and~\eqref{varrck}, the f\/irst variation of $\henon_{s/3, t}$ is
\begin{gather*}
\vr_{\Pi}\henon_{s/3, t}(\nabla) = 2\sOm_{\nabla}\big(t^{2}\hop(\K) + s^{2}\sd^{\ast}\ric, \Pi\big),
\end{gather*}
so the critical points of $\henon_{s/3, t}$ are the solutions of~\eqref{coupled}.

\subsection*{Acknowledgements}
I thank the anonymous referees for their thoughtful criticisms and detailed corrections which helped improve the article, particularly the exposition in Section~\ref{symplecticaffinesection}.

\pdfbookmark[1]{References}{ref}
\LastPageEnding

\end{document}